\documentclass[11pt,reqno]{amsart}  
\usepackage{amssymb}
\usepackage{color}
\usepackage{enumerate}
\allowdisplaybreaks[1]
\newtheorem{thm}{Theorem}[section]

\newtheorem{dfn}[thm]{Definition}
\newtheorem{prop}[thm]{Proposition}
\newtheorem{lem}[thm]{Lemma}
\newtheorem{rem}[thm]{Remark}
\renewcommand{\qed}{\qquad\kern1pt 
 \vbox{\hrule height 0.6pt      
  \hbox{\vrule width 0.6pt 
   \vbox{\vskip 6pt}  
   \hskip 3pt
  \vrule width 1.3pt} 
 \hrule depth 1.3pt}     
\kern1pt}
\newcommand{\sect}[1]{\section{#1} \setcounter{equation}{0}}

\newcommand{\pt}{\partial}           
\newcommand{\R}{\mathbb R}

\newcommand{\Z}{\mathbb Z}
\newcommand{\N}{\nabla}

\newcommand{\gm}{\gamma}
\newcommand{\lam}{\lambda}

\newcommand{\dl}{\delta}
\newcommand{\Del}{\Delta}
\newcommand{\s}{\sigma}
\renewcommand{\r}{\rho}
\renewcommand{\t}{\tau}

\newcommand{\hr}{\hookrightarrow}
\newcommand{\wt}{\widetilde }
\newcommand{\wh}{\widehat }
\renewcommand{\div}{{\rm{div\,}}}

\newcommand{\dB}{\dot{B}}
\newcommand{\cB}{\dot{B}_{2,1}}
\newcommand{\fB}{\widehat{\dot{B}}}
\newcommand{\cfB}{\widehat{\dot{B}}_{p,1}}
\newcommand{\ve}{\varepsilon}

\newcommand{\mX}{\mathcal{X}}

\newcommand{\mD}{\mathcal{D}}
\newcommand{\mG}{\mathcal{G}}

\newcommand{\al}{\alpha}
\newcommand{\supp}{\text{supp\;}}
\renewcommand{\L}{\mathcal{L}}
\renewcommand{\P}{\mathcal{P}}
\newcommand{\les}{\lesssim\;}
\newcommand{\dsp}{\displaystyle}
\newcommand{\F}{\mathcal{F}}
\newcommand{\Id}{\text{Id}} 

\begin{document}
\title{On the time-decay with the diffusion wave phenomenon of the solution to the compressible Navier-Stokes-Korteweg system in critical spaces} 
\maketitle
\vskip5mm
{\small
\hskip7mm\noindent
{\normalsize \sf Takayuki Kobayashi}
\vskip1mm\hskip1.5cm
        Graduate School of Engineering Science, The University of Osaka  \hfill\break\indent\hskip1.5cm
        Osaka 560-8531, Japan \hfill\break\indent\hskip1.5cm 
        kobayashi@sigmath.es.osaka-u.ac.jp
}
\vskip5mm
{\small
\hskip7mm\noindent
{\normalsize \sf Ryosuke Nakasato}
\vskip1mm\hskip1.5cm
        Faculty of Engineering, Shinshu University \hfill\break\indent\hskip1.5cm
        Nagano 380-8553, Japan \hfill\break\indent\hskip1.5cm
        nakasato@shinshu-u.ac.jp}

\vspace*{2mm}
\begin{abstract} 
We consider the initial value problem of the compressible Navier-Stokes-Korteweg equations in the whole space $\mathbb{R}^d$ ($d \ge 2$). 
The purposes of this paper are to obtain the global-in-time solution around the constant equilibrium states $(\rho_*,0)$ and investigate the $L^p$-$L^1$ type 
time-decay estimates in a scaling critical framework, where $\r_*>0$ is a constant.  
In addition, we study the diffusion wave property came from the wave equation with strong damping for the solution 
with the initial data belonging to the critical Besov space. 
The key idea of the proof is the derivation of the time-decay for the Fourier-Besov norm with higher derivatives 
by using $L^1$-maximal regularity for the perturbed equations around $(\r_*,0)$. 
\end{abstract}


\baselineskip=5.6mm
\sect{Introduction}

This paper investigates the global well-posedness and asymptotic behavior of solutions   
to the initial value problem of the following compressible   
Navier-Stokes-Korteweg system in the $d$-dimensional Euclidean space $\R^d$ ($d \ge 2$): 
\vskip-1.5mm
\begin{equation} \label{eqn;NSK} 
\left\{
\begin{aligned}
& \pt_t\r+\div (\r u)=0, 
&t>0, x \in \R^d, \\
&\pt_t (\r u) 
+\text{div}\left(\r u \otimes u \right)+\N P(\r)= \text{div} \left(\mathcal{T}(\r,u)+\mathcal{K}(\r)\right)
&t>0, x \in \R^d, \\
&(\r, u)|_{t=0}=(\r_0,u_0), &x \in \R^d,  
\end{aligned}
\right.
\end{equation}
\vskip1.5mm
\noindent
where $\r=\r(t,x):\R_+ \times \R^d \to \R_+$ and  
$u=u(t,x):\R_+ \times \R^d \to \R^d$ denote the density of the fluid and the velocity of the fluid, respectively. 
The pressure $P=P(\r)$ is assumed to be a {\it real analytic function} of $\r$ in a neighborhood of 
$\rho_*>0$. 
In addition, we assume that the pressure $P$ satisfies $P'(\rho_*)>0$. 
The viscous stress tensor $\mathcal{T}(\r,u)$ with the viscosity coefficients $\lam=\lam(\r)$, $\mu=\mu(\r)$ is given by 
$$
\mathcal{T}(\r,u)=2 \mu(\r) D(u)+\lam(\r) \div u \,\text{Id}. 
$$ 
Here $\text{Id}$ denotes the identity matrix which is given by $\Id=(\dl_{ij})_{ij}$ 
with $\dl_{ij}$ designating the Kronecker delta and the deformation tensor $D(u)$ is defined by  
$$
D(u)=\frac{1}{2}(\N u+{}^t(\N u)) \;\;\text{with}\;\;(\N u)_{ij}=\pt_i u_j.  
$$
Throughout this paper,  we assume that $\mu$, $\lam$ are given constants satisfying the standard elliptically conditions $\mu>0$ and $\nu:=\lam+2\mu>0$. 
For  $\kappa=\kappa(\r)$, 
the Korteweg stress tensor $\mathcal{K}(\r)$ in the second equation of \eqref{eqn;NSK} is given by 
\begin{equation} \label{eqn;Korteweg}
\mathcal{K}(\r)=\frac{\kappa}{2}(\Delta \r^2-|\N \r|^2) \,\text{Id} -\kappa \N \r \otimes \N \r, 
\end{equation}
where $\N \r \otimes \N \r$ stands for the tensor product $(\pt_i \r \pt_j \r)_{ij}$.  In general, it is natural that $\kappa$ depends on $\r$. 
For simplicity, in this paper, we assume that $\kappa$ is a positive constant.  

The compressible Navier-Stokes-Korteweg system \eqref{eqn;NSK} is well-known as the Diffuse Interface (DI) model describes the motion of 
a vaper-liquid mixture in a compressible viscous fluid. 
The primitive theory regarding the DI model was first proposed by Van der Waals \cite{Va}. Later, Korteweg \cite{Ko} introduced the stress tensor 
including the term $\N \r \otimes \N \r$. 
The rigorous derivation of the corresponding equations \eqref{eqn;NSK} was given by Dunn-Serrin \cite{Du-Se}. 

For the existence of a strong solution to the problem \eqref{eqn;NSK}, Hattori-Li \cite{Ha-Li,Ha-Li2} established in a inhomogeneous Sobolev framework.   
The one of the purpose of this paper is to obtain a global-in-time solution of the problem \eqref{eqn;NSK}  
in a scaling {\it critical} $L^p(\R^d)$ framework. In order to explain what we mean by critical regularity, 
we focus on a nearly scaling invariance property for the compressible Navier-Stokes-Korteweg system. 
For a solution $(\r,u)$ to the problem \eqref{eqn;NSK},  
the scaled functions $(\r_\al,u_\al)$ given by 
\begin{equation} \label{eqn;c1}
\left\{
\begin{aligned}
 &\r_\al(t,x)= \r(\al^2t,\al x), \\
 &u_\al(t,x)= \al u(\al^2 t, \al x), 
\end{aligned}
\right. \quad \al >0
\end{equation}
also satisfy the same problem without the initial conditions provided that the pressure $P$ is changed into $\al^2 P$.   
By virtue of Fujita-Kato's principle \cite{Fj-Kt}, one can find that 
the invariant function class under the scaling \eqref{eqn;c1} is given by 
{\it critical function spaces}.  
The critical space for the initial data $(\r_0,u_0)$ in problem \eqref{eqn;NSK}, 
for instance, is given by the homogeneous Besov space   
$\dB^{d/p}_{p,\s}(\R^d) \times \dB^{-1+d/p}_{p,\s}(\R^d)$ 
(for the definition of Besov spaces, see Definition 2.15 in \cite{B-C-D}). 
Focusing on the above nearly scaling invariances, Danchin-Desjardins \cite{Da-De} established the result on the small data global existence 
of the problem \eqref{eqn;NSK} in the critical Besov space based on $L^2(\R^d)$. 
Recently, Charve-Danchin-Xu \cite{Ch-Da-Xu} obtained the global well-posedness and Gevrey analyticity for the global solution in more 
general critical $L^p(\R^d)$ framework. As for the another interesting result, see \cite{Has, Zh}. 

Regarding the asymptotic stability for the global solution to the problem \eqref{eqn;NSK} around the constant equilibrium state $(\r_*,0)$, 
there are numerous studies up to the present. 
In the case of the smooth initial data, the many authors have established the $L^2$-$L^1$ type time decay estimate, 
which is the same decay rate as the one of the $L^2(\R^d)$-norm of the fundamental solution to the heat equation provided the initial data belongs to $L^1(\R^d)$ (cf. \cite{Ta-Wa-Xu,Ta-Zh,Wa-Ta,Zh-Ta}). 
In the scaling critical case, Chikami-Kobayashi \cite{Ch-Ko} showed that the $\dot{B}_{2,1}^s$-$\dot{B}_{2,\infty}^{-d/2}$ type time-decay estimate holds true 
for the global solution obtained by Danchin-Desjardins \cite{Da-De}. They also handled the global well-posedness in critical Besov spaces 
and the time-decay of solutions in the case of the zero sound speed $P'(\r_*)=0$. 
Later, Kawashima-Shibata-Xu \cite{Ka-Sh-Xu} extended the Chikami-Kobayashi's decay result to the critical $L^p(\R^d)$-Besov framework. 
Inspiring by these studies, in another scaling critical framework, we investigate the asymptotic stability and the more precise asymptotic behavior of solutions as handled in Kawashima-Matsumura-Nishida 
\cite{Ka-Ma-Ni} and Hoff-Zumbrun \cite{Ho-Zu,Ho-Zu2}.   

\sect{Reformulation and Main Results} 
In this paper, we focus on a solution for the density that is close to a constant equilibrium state  
$\rho_*>0$ at spatial infinity. Setting $m:=\r u$,  let us reformulate the problem \eqref{eqn;NSK} as followings:  
\begin{equation} \label{eqn;mNSK} 
\left\{
\begin{aligned}
& \pt_t\r+\div m=0, 
&t>0, x \in \R^d, \\
&\pt_t m 
+\text{div}\left(\r^{-1} m \otimes m \right)+\N P(\r)= \mathcal{L}\left(\r^{-1}m \right)+\text{div} \left(\mathcal{K}(\r)\right)
&t>0, x \in \R^d, \\
&(\r, m)|_{t=0}=(\r_0,m_0), &x \in \R^d. 
\end{aligned}
\right.
\end{equation}
In what follows, we perform our analysis in the above momentum representation \eqref{eqn;mNSK}. 
The merit of the momentum representation is to maintain the divergence form, which the standard velocity formulation does not equip. 
Taking into account the nonlinear terms equipped with the divergence form, 
we may be possible to obtain better results than the previous studies on the time-decay and analyze the asymptotic behavior of solutions with the critical regularity.  

\subsection{Definition of Function Spaces}
Before stating the main statements of this paper, we introduce some 
notation and definitions. 
For $d \ge 1$ and $1\le p\le\infty$, let $L^p = L^p(\R^d)$ be the Lebesgue space. 
For any $f$ belonging to the {\it Schwartz class} $\mathcal{S}=\mathcal{S}(\R^d)$, 
the Fourier transform of $f$ denoted by $\wh{f}=\wh{f}(\xi)$ or $\mathcal{F}[f]=\mathcal{F}[f](\xi)$ is  
$$
 \mathcal{F}[f](\xi)(=\wh{f}(\xi)):= \frac{1}{(2 \pi)^\frac{d}{2}} \int_{\R^d} e^{-i x \cdot \xi}f(x) \,dx.             
$$
Similarly, 
for any $g=g(\xi)$ belonging to $\mathcal{S}(\R^d_\xi)$, the Fourier {\it inverse} transform $\mathcal{F}^{-1}[g]=\mathcal{F}^{-1}[g](x)$ is then defined as 
$$
 \mathcal{F}^{-1}[g](x) := \frac{1}{(2 \pi)^\frac{d}{2}} \int_{\R^d} e^{i x \cdot \xi} g(\xi) \,d\xi.             
$$
Let $\{ \phi_j \}_{j \in \Z}$ be the Littlewood-Paley 
dyadic decomposition of unity, i.e., for a non-negative radially 
symmetric function $\phi \in \mathcal{S}$, we set (for the construction 
of $\{ \phi_j \}_{j \in \Z}$, see e.g., \cite{B-C-D}, \cite{S}) 
\begin{equation*}
\wh{\phi_j}(\xi) := \wh{\phi}(2^{-j}\xi) \;(j \in \Z), \;\;
\sum_{j \in \Z}\wh{\phi}_j(\xi) = 1 \; (\xi \neq 0) \;\;
\text{and} \;\;
\supp \wh{\phi}
\subset
\{\xi \in \R^{d};\frac{1}{2} \le |\xi| \le 2\}. 
\end{equation*}

Let us introduce Fourier-Besov spaces. 
The definition below is inspired by Gr\"{u}nrock \cite{Gr}. 
\vskip2mm

\begin{dfn}[{\it Homogeneous Fourier-Besov spaces}\hspace{0mm}] Let $s \in \R$, $d \ge 1$, $1 \le p,\s \le \infty$ 
and $\mathcal{S}'=\mathcal{S}'(\R^d)$ be the space of tempered distributions. 
We define the homogeneous Fourier--Besov space $\fB_{p,\s}^s = \fB_{p,\s}^s(\R^d)$ as follows:
\begin{align*}
  \fB_{p,\s}^s(\R^d)
  :=\{f\in\mathcal{S}';\widehat{f}\in L^1_{loc}(\R^d), \,\|f\|_{\fB_{p,\s}^s} < \infty\}, \quad
  \|f\|_{\fB_{p,\s}^s}:=\left\|\left\{2^{sj}\|\wh{\phi}_j\wh{f}\|_{L^{p'}}\right\}_{j\in\Z}\right\|_{\ell^\s},
\end{align*} 
where $p'$ is the exponent of H\"older conjugate, namely $1/p+1/{p'}=1$. 
\end{dfn}
\begin{rem}[]
Let $s \in \R$, $d \ge 1$ and $1 \le \s \le \infty$. By virtue of Plancherel's theorem, we see that  
the spaces $\fB_{2,\s}^s$ and $\dB_{2,\s}^s$ coincide in the sence of norm equivalence, 
where and in what follows, $\dB_{2,\s}^s=\dB_{2,\s}^s(\R^d)$ is the homogeneous Besov space defined by 
$$
  \dB^s_{2,\s}:=\{f \in \mathcal{S}';\wh{f}\in L^1_{loc}(\R^d), \,\|f\|_{\dB^s_{2,\s}}<\infty\}, 
  \quad 
  \|f\|_{\dB^s_{2,\s}}:=\left\|\left\{2^{sj}\|\phi_j*f\|_{L^2}\right\}_{j\in\Z}\right\|_{\ell^\s}.  
$$  
\end{rem}
\vskip2mm
Taking into account the time variable, we give the definition of the space-time mixed spaces 
introduced by Chemin-Lerner \cite{Ch-Le}.
\vskip2mm
\begin{dfn}[{\it Chemin-Lerner spaces}\hspace{0mm}] 
Let $s \in \R$, $d \ge 1$, $1 \le p \le \infty$, $1 \le \s,r \le \infty$ and $T \in \R_+$.  
We define the Chemin--Lerner space based on the Fourier--Besov space as follows:
\begin{align*}
  \wt{L^r(0,\,}T;\fB_{p,\s}^s)
  :=\overline{C(0,T;\mathcal{S}_0)}^{\|\cdot\|_{\wt{L^r(0,\,}T;\fB_{p,\s}^s)}}, \;\;
  \|f\|_{\wt{L^r(0,\,}T;\fB_{p,\s}^s)}:=\left\|\left\{2^{sj}\|\wh{\phi}_j \wh{f}\|_{L^r(0,T;L^{p'})}\right\}_{j\in\Z}\right\|_{\ell^\s}, 
\end{align*}
where $\mathcal{S}_0=\mathcal{S}_0(\R^d)$ is the set of functions in the Schwartz class 
$\mathcal{S}(\R^d)$  
whose Fourier transform are supported away from $0$. 
\end{dfn}
\begin{rem} 
By Minkowski's inequality, we note that the following continuous embeddings hold between the 
Chemin--Lerner space $\wt{L^r(0,\,}T;\fB_{p,\s}^s)$ and the Bochner space 
$L^r(0,T;\fB_{p,\s}^s)$:
\begin{align*}
      \wt{L^r(0,\,}T;\fB_{p,\s}^s) 
       \hookrightarrow L^r(0,T;\fB_{p,\s}^s) \quad \text{if}\;\; r \ge \s, \quad
      L^r(0,T;\fB_{p,\s}^s) 
       \hookrightarrow \wt{L^r(0,\,}T;\fB_{p,\s}^s) \quad \text{if}\;\; r \le \s 
\end{align*}
(for the definition of Bochner spaces, see the end of this section). 
\end{rem}

\subsection{Global Well-posedness and Decay Estimates}
In the followings, we state the first half of our main results. 
First of all, we state the result on the global well-posedness and $L^p$-$L^1$ type time-decay estimates of global solutions for 
the initial value problem \eqref{eqn;mNSK} in the critical Fourier-Besov spaces. 

\begin{thm}[{\it Global well-posedness}\hspace{0mm}] \label{thm;GWP_FLp}
Let $d \ge 2$ and $1 \le p \le \infty$. Suppose that $(\r_0,m_0)$ satisfy 
$$
(\r_0-\r_*, m_0)
\in \bigg(\cfB^{-1+\frac{d}{p}}\cap \cfB^{\frac{d}{p}}\bigg)(\R^d) \times \cfB^{-1+\frac{d}{p}}(\R^d). 
$$
There exists a positive constant $\ve_0 \ll 1$ such that if 
\begin{equation} \label{assump;initial_FB}
\mX_{p,0}:= \|\r_0-\r_*\|_{\cfB^{-1+\frac{d}{p}}\cap \cfB^{\frac{d}{p}}}
            +\|m_0\|_{\cfB^{-1+\frac{d}{p}}}
\le \ve_0,  
\end{equation}
then the problem \eqref{eqn;mNSK} admits a unique global-in-time solution $(\r,m)$ satisfying 
\begin{gather*}
(\langle \N \rangle(\r-\r_*),m) \in C([0,\infty);\cfB^{-1+\frac{d}{p}}) \cap L^1(\R_+;\cfB^{1+\frac{d}{p}}), 
\end{gather*}
where $\langle \N\rangle$ is the Bessel potential given by $\mathcal{F}^{-1}\langle \xi \rangle\mathcal{F}$ with $\langle\xi\rangle=\sqrt{1+|\xi|^2}$. 
In particular, there exists some constant $C>0$ such that for all $t \ge 0$, 
\begin{equation} \label{est;unif}
\begin{aligned}
\mX_p(t):=
\|(\langle\N\rangle(\r-\r_*),m)\|_{\wt{\;L^\infty(0},t;\cfB^{-1+\frac{d}{p}})\cap L^1(0,t;\cfB^{1+\frac{d}{p}})} 
\le C \mX_{p,0}. 
\end{aligned}
\end{equation}
\end{thm}

\begin{thm}[{\it $L^p$-$L^1$ type decay estimates}\hspace{1mm}] \label{thm;Lp-L1}
Let $1 \le p \le \frac{2d}{d-1}$. Let the initial data $(\r_0,m_0)$ 
satisfy the same assumption as in Theorem \ref{thm;GWP_FLp} and 
$(\r,m)$ denote the corresponding global-in-time solution of the problem \eqref{eqn;mNSK}. 
If in addition 
\begin{align} \label{cond;add_small}
\mathcal{D}_{p,0}:= \sup_{j \le j_0} 2^{-\frac{d}{p'}j}\|\phi_j*(\r_0-\r_*,m_0)\|_{\wh{L}^p} \le \ve_0, 
\end{align}
then the global solution $(\r,m)$ satisfies the following decay estimates for all $t \ge 1$, 
\begin{equation} \label{est;decay}
\begin{aligned}
&\||\N|^{s}(\r-\r_*,m)(t)\|_{\cfB^0} \le C\big(\mX_{p,0}+\mD_{p,0}\big)t^{-\frac{d}{2}(1-\frac{1}{p})-\frac{s}{2}}, \quad 
-\frac{d}{p'}<s\le1+\frac{d}{p},  
\end{aligned}
\end{equation}
where $|\N|^s$ is Riesz's pseudo-differential operator defined by $\mathcal{F}^{-1}|\xi|^s\mathcal{F}$. 
In particular, if $p \le 2$, then the following estimate concerning the linear approximation holds  
for all $-d/p'<s \le d/p$:  
\begin{equation} \label{est;lin_approx}
\left\|
\left(
\begin{array}{@{\,}c@{\,}} 
\r-\r_*\\
m
\end{array}
\right)(t)
-G(t,\cdot)*
\left(
\begin{array}{@{\,}c@{\,}} 
\r_0-\r_* \\
m_0
\end{array}
\right)
\right\|_{\cfB^s}
=O\left(t^{-\frac{d}{2}(1-\frac{1}{p})-\frac{s+1}{2}} \delta(t) \right) \;\;(t \to \infty), 
\end{equation}
where $\delta(t)$ is given by $\delta(t)=\log t$ if $d=2$ and $\delta(t)=1$ if $d \ge 3$ and the definition of the Green matrix $G(t,x)$ is given in \S\,\ref{sect;linear}. 
\end{thm}

\noindent
{\bf Comments on Theorems \ref{thm;GWP_FLp}, \ref{thm;Lp-L1}.} 
\begin{enumerate}[1.]
\item 
(Global well-posedness) 
As is well-known in the multi-dimensional case, the solvability of the compressible fluid model requires 
the low-frequency of the initial data assumed to be the critical $L^2(\R^d)$-regularity because of the effect on the dispersive operator likes $e^{it|\N|}$ 
(cf. \cite{Ch-Da2, Ch-Da-Xu, Ch-Mi-Zh}). 
However, Theorem \ref{thm;GWP_FLp} asserts that the low-frequency of the initial data 
should not satisfy the $L^2(\R^d)$-regularity and 
the range of the integrable index $p$ is possible to take as far as the end-point $p=\infty$,  
because we consider the solvability of the problem in the critical space based on the Herz type space $\wh{L}^p(\R^d)$ 
(for the definition, see Lemma \ref{lem;Bern} below), 
which guarantees the boundedness of $e^{it|\N|}$.  

The choice of $p=\infty$ makes the function class largest among the Fourier-Besov setting and this is the completely same as the function space 
$\mathcal{X}^{-1}(\R^d)=\{f \in \mathcal{S}'(\R^d);\int_{\R^d}|\xi|^{-1}|\wh{f}(\xi)|\,d\xi<+\infty\}$ introduced by Lei-Lin \cite{Le-Li} (see Lemma \ref{lem;equi} below).  
To the best of author's knowledge, 
Theorem \ref{thm;GWP_FLp} is the first result on the existence of global solutions to the compressible viscous model 
in the critical Fourier-Besov space with the end-point $p=\infty$.  

\item (Time-decay estimates)  
The assertions of Theorem \ref{thm;Lp-L1} 
means that the $L^p$-$L^{1}$ type decay estimate holds true for the global solution with the critical regularity obtained by Theorem \ref{thm;GWP_FLp}.  
By virtue of the pointwise estimate for the linearized problem \eqref{eqn;L} in Lemma \ref{lem;pw_L}, 
one can see that the decay rate of solution is the same one as the heat kernel. 
In the proof of Theorem \ref{thm;Lp-L1} in \S\,\ref{sect;Lp-L1}, 
applying the decay estimates for the linearized solution as developed by Kozono-Ogawa-Taniuchi \cite{Ko-Og-Ta} 
and the refined Fourier splitting methods to the solution of \eqref{eqn;mNSK}, 
we are possible to get the time-decay of $\wh{L}^p(\R^d)$-norm 
when the low frequencies of initial data is assumed to be the assumption \eqref{cond;add_small} weaker than 
$L^1(\R^d)$. Here we notice that the following continuous embeddings holds true for any 
$1 \le p \le \infty$, 
$$
L^1(\R^d) \hr 
\widehat{L}^1(\R^d) \simeq \widehat{\dB}_{1,\infty}^0(\R^d) \hr \fB_{p,\infty}^{-\frac{d}{p'}}(\R^d) 
$$
(for the proofs, Lemmas \ref{lem;sm} and \ref{lem;equi} below). 

In regards to the result on the linear approximation  \eqref{est;lin_approx}, the same estimate for smooth solutions in the $L^2(\R^3)$ sense is firstly developed by 
Kawashima-Matsumura-Nishida \cite{Ka-Ma-Ni}. 
Later, Hoff-Zumbrun \cite{Ho-Zu,Ho-Zu2} extended their result to the general $L^p(\R^d)$ space. 
As the asymptotic behavior for the compressible fluid model in a scaling critical framework, the estimate \eqref{est;lin_approx} in Theorem \ref{thm;Lp-L1} is the first result.  
\end{enumerate}

\subsection{Diffusion Wave Property}
As a phenomenon specific the compressible fluid models, it is well-known that the time-decay of the compressible part of the linearized solution 
is influenced by the dispersive effect, which is similar to the fundamental solution of the wave equation. 
Here, the compressible part of some vector valued function $v \in \mathcal{S}'(\R^d)^d$ is represented by $(-\Delta)^{-1} \N \div v$.  
This property is often called the ``{\it diffusion wave property}''. 
Starting with the pioneering work by Hoff-Zumbrun \cite{Ho-Zu,Ho-Zu2}, Shibata \cite{Sh} and Kobayashi-Shibata \cite{Ko-Sh} established 
the refined time-decay estimate for the linearized solution of the compressible Navier-Stokes equations. 
Roughly speaking, they showed that the compressible part of the  linearized solution behaves like $e^{t \nu \Delta} e^{it \gm |\N|} U_0$ 
with some initial data $U_0$. 
In regard to the compressible Navier-Stokes-Korteweg system, Kobayashi-Tsuda \cite{Ko-Ts} derived the decay estimate with the diffusion wave property for  
smooth solutions. 
On the other hand, regarding the optimality for the decay rate of the $L^\infty(\R^3)$-norm of the compressible part of the linearized solution 
when the $L^1(\R^3)$-norm of the initial data is bounded, we refer to the recent work by Iwabuchi-\'O hAodha \cite{Iw-Da}. 

To the best of author's knowledge, the decay estimate with the diffusion wave property has not been obtained in scaling critical frameworks. 
In this paper, we focus on the proof of the asymptotics involving its property of the global solution as obtained by Theorem \ref{thm;GWP_FLp}. 
Restricting the integrable index $p=2$, we obtain the following result. 
\begin{thm}[{\it Diffusion wave phenomenon}] \label{thm;asympt}
Let $p=2$ and the initial data $(\r_0,m_0)$ 
satisfy the same assumption as in Theorem \ref{thm;GWP_FLp} and 
$(\r,m)$ denote the corresponding global-in-time solution of the problem \eqref{eqn;mNSK}. 
If in addition 
\begin{align} \label{cond;L1}
\|(\r_0-\r_*,m_0)\|_{L^1} \le \ve_0, 
\end{align}
then the global solution $(\r,m)$ satisfies the following decay estimates: 
\begin{equation} \label{est;curl_free}
\left\|
\left(
\begin{array}{@{\,}c@{\,}}
\r-\r_*\\
m
\end{array}
\right)(t)
-
\left(
\begin{array}{@{\,}c@{\,}}
0 \\
e^{t \mu \Del}\P_\s m_0
\end{array}
\right)
\right\|_{L^\infty}
=O\left(t^{-\min\left(\frac{3d-1}{4},\frac{d}{2}+\frac{1}{2}\right)}\right) \;\;(t \to \infty). 
\end{equation}
Here $e^{t \mu \Delta} f:=\mathcal{F}^{-1}[e^{-t \mu |\xi|^2}\wh{f}\,]$  
and $\mathcal{P}_\s v:=v+(-\Delta)^{-1} \N \div v$ for $f \in \mathcal{S}'(\R^d)$, $v \in \mathcal{S}'(\R^d)^d$. 
\end{thm}

\noindent
{\bf Comments on Theorem \ref{thm;asympt}.}  
The first half of the decay rate $(3d-1)/4$ in \eqref{est;curl_free} is came from the compressible part of the linearized solution $U_L(t)$ as given by \eqref{eqn;mg}. 
Let us denote it by $U_{c,L}(t)=U_L(t)-{}^t(0,e^{t \mu \Delta}\mathcal{P}_\s m_0)$.  
It is obvious that $e^{t \mu \Delta}\mathcal{P}_\s m_0$ is the solution to the Stokes equation 
in whole space $\R^d$. 
Since we can identify $U_{c,L}(t)$ as $U_{c,L}(t) \simeq e^{t \nu\Delta} e^{it \gm|\N|} U_0$ by applying the Taylor expansion to $\lam_{\pm}(\xi)$ in \eqref{eqn;mg}, 
for some smooth function $f$, it follows from the change of variables 
$\xi \mapsto \xi/\sqrt{t}$ that 
$$
e^{t \nu \Delta} e^{it \gm |\N|} f
= \F^{-1}\Big[e^{-t \nu |\xi|^2}e^{it \gm |\xi|}\wh{f}\,\Big] 
= t^{-\frac{d}{2}} \F^{-1}\Big[e^{-\nu |\xi|^2} e^{i\sqrt{t}\gm |\xi|}\wh{f}\,\Big], 
$$
where $\gm:=\sqrt{P'(\r_*)}>0$. 
By directly using the Hausdorff-Young inequality and $L^\infty(\R^d)$-$L^1(\R^d)$ estimate with $d \ge 2$ for the fundamental solution of the wave equation 
(see e.g., Brenner \cite{Br}), it is straightforward to obtain the first half of decay rate $(3d-1)/4$ as $d/2+(d-1)/4$. 
On the other hand, as is shown in \eqref{est;lin_approx} of Theorem \ref{thm;Lp-L1}, 
the second half of the decay rate $d/2+1/2$ is came from the nonlinear terms.      

\vskip2mm
This paper is organizsed as follows: After preparation on the elementary inequalities in the Fourier-Besov spaces 
and introducing some estimates for the solution of the linearized problem \eqref{eqn;L} 
in \S \,\ref{sect;prelimiary} and \S \,\ref{sect;linear}, we give 
the proof of Theorem \ref{thm;GWP_FLp} in \S \,\ref{sect;GWP}. 
In Section \ref{sect;Lp-L1} , we show the $L^p$-$L^{1}$ type decay estimates and the linear approximations for 
the global solution to the problem \eqref{eqn;mNSK}. 
Section \ref{sect;diff_wave} is devoted to the proof of Theorem \ref{thm;asympt}. 

\vskip2mm 
\noindent
{\bf Notation.}\;Throughout this paper, $C$ stands for a generic constant 
(the meaning of which depends on the context). 
Let $X$ be a Banach space, $I \subset \R$ be an interval,  
and $1\le r \le \infty$. One denotes by the Bochner space   
$L^r(I;X)$ the set of strongly measurable functions $f:I\to X$ such that $t \mapsto \|f(t)\|_X$ 
belongs to $L^r(I)$. For $f \in L^r(I;X)$, one defines $\|f\|_{L^r(I;X)}=\|\|f\|_{X}\|_{L^r(I)}$.   
For simplicity, we denote the Chemin--Lerner space $\wt{L^r(0,T;}\fB_{p,q}^s)$ 
by $\wt{L^r_T}(\fB_{p,q}^s)$. 
For $s \in \R$ and $f \in \mathcal{S}'$, $|\N|^s$ is designating the Riesz potential defined by $|\N|^s f:=\mathcal{F}^{-1}[|\xi|^s \wh{f}\,]$ and 
$\langle\N\rangle^s f:=\mathcal{F}^{-1}[\langle\xi\rangle^s \wh{f}]$ with $\langle\xi\rangle^s:=(1+|\xi|^2)^{s/2}$. 
For any normed space $Y$, we denote by 
$\|(f,g)\|_Y := (\|f\|_Y^2 + \|g\|_Y^2)^{1/2}$ with $f$, $g \in Y$. 
In what follows, $f_j:=\phi_j*f$, $S_{j}f:= \sum_{k \le j-1} f_k$, 
$\tilde{\phi}_j:=\sum_{|j-k|\le 1} \phi_k$ and $\tilde{f}_j:= \tilde{\phi}_j*f$. 
For $1 \le q<\infty$, $\ell^q(\Z)$ denotes the set of all sequences $\{a_j\}_{j \in \Z}$ of real numbers such that $\sum_{j \in \Z}|a_j|^q$ converges. 
In the case $q=\infty$, let $\ell^\infty(\Z)$ denote the set of all bounded sequences of real numbers.  
For some $j_0 \in \mathbb{Z}$, we set 
$u^\ell := \sum_{j \le j_0-1} \phi_j * u$ and 
$u^h := u - u^\ell$. For $s \in \R$ and $1 \le p,r \le \infty$, we also use the following norm restricted to  low-and high-frequency parts: 
\begin{align*}
    &\|u\|_{\cfB^s}^{\ell} := \sum_{j \le j_0} 2^{sj} \|\wh{\phi}_j \wh{u}\|_{L^{p'}} \quad \text{and} \quad 
     \|u\|_{\cfB^s}^h := \sum_{j \ge j_0-1} 2^{sj} \|\wh{\phi}_j \wh{u}\|_{L^{p'}}, \\
    &\|u\|_{\wt{L^r(0,\,}T;\cfB^s)}^\ell
      := \sum_{j \le j_0} 2^{sj} \|\wh{\phi}_j \wh{u}\|_{L^r(0,T;L^{p'})} \quad \text{and} \\
    &\|u\|_{\wt{L^r(0,\,}T;\cfB^s)}^h
      := \sum_{j \ge j_0 - 1} 2^{sj} \|\wh{\phi}_j \wh{u}\|_{L^r(0,T;L^{p'})}.  
\end{align*}
From the definition of the above norms, we easily see that for all $s_1 \le s_2$, $1 \le p \le \infty$, 
$$\|u\|_{\cfB^{s_2}}^\ell \les \|u\|_{\cfB^{s_1}}^\ell, \quad  
\|u\|_{\cfB^{s_1}}^h \les \|u\|_{\cfB^{s_2}}^h. 
$$  
The analogous estimates hold true for the Chimin-Lerner type norms.

\sect{Preliminaries}\label{sect;prelimiary}

In the followings, we present the basic properties with regards to 
the Fourier--Besov spaces. First of all, let us now recall Bernstein's inequalities   
which allows us to obtain some embedding of spaces (for the proofs, see Lemma A.1 in \cite{Ch2}).   
\begin{lem}[{\it Bernstein-type lemma} \cite{Ch2}] \label{lem;Bern}
Let $d \ge 1$, $1 \le p \le q \le \infty$, $\lam$, $R >0$ and $0 < R_1 < R_2$.  
For any $s \in \R$, there exists some constant $C>0$ such that 
\begin{align*}
  &\||\N|^s f\|_{\wh{L}^q}
  \le C \lam^{s + d(\frac{1}{p}-\frac{1}{q})}\|f\|_{\wh{L}^p} \quad 
  if \;\; \supp \wh{f} \subset \{\xi \in \R^d;|\xi| \le \lam R\},\\
  &C^{-1} \lam^s \|f\|_{\wh{L}^p}
  \le \||\N|^s f\|_{\wh{L}^p}
  \le C \lam^s \|f\|_{\wh{L}^p} \quad 
  if \;\; \supp \wh{f} \subset \{\xi \in \R^d;\lam R_1 \le |x| \le \lam R_2\}, 
\end{align*}
where the Fourier--Lebesgue space $\wh{L}^p=\wh{L}^p(\R^d)$ is defined by
$$
  \wh{L}^p(\R^d)=\{f \in \mathcal{S}'; \wh{f} \in L^1_{loc}(\R^d),\|f\|_{\wh{L}^p} <\infty\} \quad \text{with} \quad
  \|f\|_{\wh{L}^p}=\|\wh{f}\|_{L^{p'}}. 
$$
\end{lem}

\begin{lem}[{\it Sobolev-type embedding \cite{Na}}\hspace{0mm}] \label{lem;sm}
Let $s \in \R$, $d \ge 1$, $1 \le p_1 \le p_2 \le \infty$ and $1 \le \s_1\le\s_2 \le \infty$.  
Then the following continuous embedding holds:
\begin{align*}
  \fB_{p_1,\s_1}^s(\R^d) \hr \fB_{p_2,\s_2}^{s-d(\frac{1}{p_1}-\frac{1}{p_2})}(\R^d). 
\end{align*}
\end{lem}

\begin{lem}[\cite{Ch2,Ko-Yo}] \label{lem;equi}
Let $s \in \R$, $d \ge 1$ and $1 \le p \le \infty$. The spaces $\fB_{p,p'}^s(\R^d)$ and 
$\wh{\dot{H}}_p^s(\R^d)$ satisfy 
$
\fB_{p,p'}^s(\R^d) \simeq \wh{\dot{H}}_p^s(\R^d)
$ 
in the sence of norm equivalence. 
Here the Fourier--Sobolev space $\wh{\dot{H}}_p^s=\wh{\dot{H}}_p^s(\R^d)$ is defined by 
$$
  \wh{\dot{H}}_p^s(\R^d)
  =\{f \in \mathcal{S}'; \wh{f} \in L^1_{loc}(\R^d),\|f\|_{\wh{\dot{H}}_p^s}<\infty\} \quad \text{with} \quad
  \|f\|_{\wh{\dot{H}}_p^s}=\||\N|^sf\|_{\wh{L}^p}. 
$$
\end{lem}

\subsection{Bilinear estimates and product estimates}
Let us present the various types of product estimates in 
Fourier--Besov spaces or Chemin--Lerner type spaces. 
In the following lemmas, we state the standard bilinear estimates without their proof 
(for the general statement and their proof, see e.g., Lemma 2.4 in \cite{Ma-Na-Og}).  
\begin{lem}[{\it Bilinear estimates}\hspace{0mm}] \label{lem;bil}
Let $s > 0$ and $1 \le p,\s \le \infty$. 
There exists some constant $C > 0$ such that the following estimate holds: 
\begin{equation} \label{est;prod_bl}
\|fg\|_{\fB_{p,\s}^s}
\le 
C\left(\|f\|_{\wh{L}^\infty}\|g\|_{\fB^s_{p,\s}} 
+ \|f\|_{\fB^s_{p,\s}}\|g\|_{\wh{L}^\infty}\right). 
\end{equation}
\end{lem}   

The following product estimate in Besov spaces is well-known as in \cite{Ab-Pa}. 
The similar estimate in Fourier--Besov spaces is established in \cite{Na}. 
\begin{lem}[{\it Product estimates} \cite{Na}] \label{lem;A-P}
Let $d \ge 1$ and $1 \le p,\s \le \infty$.  
If $s \in \R$ satisfying $|s|<\frac{d}{p}$ for $2 \le p$ and $-\frac{d}{p'}<s<\frac{d}{p}$ for $1\le p<2$, 
then there exists some constant $C>0$ such that 
\begin{equation} \label{est;prod_A-P2}
\|fg\|_{\fB_{p,\s}^{s}}
\le C\|f\|_{\fB_{p,\s}^s} 
     \|g\|_{\wh{L}^{\infty}\cap\fB_{p,\infty}^\frac{d}{p}}. 
\end{equation}
\end{lem}

\begin{lem}[{\it Product estimates in Chemin--Lerner spaces} \cite{Na}] \label{lem;prod_est}
Let $s \in \R$, $1 \le p,\s \le \infty$, $I=(0,T)$ with $T \in \R_+$, 
$1 \le r,r_i \le \infty$ $(i=1,2,3,4)$ satisfying 
$\frac{1}{r}=\frac{1}{r_1}+\frac{1}{r_2}=\frac{1}{r_3}+\frac{1}{r_4}$.  
\begin{enumerate}
\item If $s>0$, there exists some constant $C>0$ such that 
\begin{equation} \label{est;prod_1-1}
\|fg\|_{\wt{L^r(I};\fB_{p,\s}^{s})}
\le C\left(\|f\|_{\wt{L^{r_1}(I};\fB_{p,\s}^s)} 
           \|g\|_{L^{r_2}(I;\wh{L}^{\infty})}
         +\|f\|_{L^{r_3}(I;\wh{L}^\infty)} 
          \|g\|_{\wt{L^{r_4}(I};\fB_{p,\s}^s)}\right). 
\end{equation}
\item 
If $s \in \R$ satisfying $|s|<\frac{d}{p}$ for $2 \le p$ and $-\frac{d}{p'}<s<\frac{d}{p}$ for $1\le p<2$, 
then there exists some constant $C>0$ such that 
\begin{equation} \label{est;prod_1-2}
\|fg\|_{\wt{L^r(I};\fB_{p,\s}^{s})}
\le C\|f\|_{\wt{L^{r_1}(I};\fB_{p,\s}^s)} 
     \|g\|_{ L^{r_2}(I;\wh{L}^{\infty}) \cap \wt{L^{r_2}(I};\fB_{p,\infty}^\frac{d}{p})}. 
\end{equation}
\end{enumerate}
\end{lem}

\begin{rem}
The proof of Lemma \ref{lem;prod_est} (1) is same as the one of Lemma 2.4 in \cite{Ma-Na-Og}. 
By \eqref{est;prod_bl}, \eqref{est;prod_1-1} and embeddings 
$\cfB^{d/p}(\R^d) \hr \fB_{\infty,1}^0(\R^d) \simeq \wh{L}^\infty(\R^d)$, 
we obtain 
\begin{align} \label{est;Banach_ring}
\|fg\|_{\cfB^{\frac{d}{p}}}
\les \|f\|_{\cfB^\frac{d}{p}}\|g\|_{\cfB^\frac{d}{p}}, \quad 
\|fg\|_{\wt{L^\infty(I};\cfB^\frac{d}{p})}
\les \|f\|_{\wt{L^\infty(I};\cfB^\frac{d}{p})}\|g\|_{\wt{L^\infty(I};\cfB^\frac{d}{p})}. 
\end{align}
\end{rem}

\subsection{Smoothing properties in Chemin--Lerner spaces}
Let us consider the following inhomogeneous heat equation with a diffusive coefficient $\nu > 0$:
\begin{equation} \label{eqn;heat}
\left\{
\begin{aligned}
  &\pt_t u - \nu \Delta u = f, &t>0, \,x \in \R^d, \\
  &u|_{t=0} = u_0, &x \in \R^d.  
\end{aligned}
\right.
\end{equation}

\begin{lem}[{\it Smoothing estimates for} \eqref{eqn;heat} \cite{Ma-Na-Og,Na}] \label{lem;MR}
Let $s \in \R$, $d \ge 1$, $1 \le p,\s \le \infty$, $1 \le r_1 \le r \le \infty$ and $I:=(0,T)$ with $T \in \R_+$. 
Suppose that $u_0 \in \fB^s_{p,\s}$ and $f \in \wt{\,L^{r_1}(I};\wh{\dB}_{p,\s}^{s-2+2/r_1})$. 
The inhomogeneous heat equation \eqref{eqn;heat} has a unique solution $u$ and 
there exists some constant $C = C(r)>0$ such that 
\begin{equation*} 
\nu^\frac{1}{r} \|u\|_{\wt{L^r(I;}\wh{\dB}_{p,\s}^{s+\frac{2}{r}})}
\le C \left(\|u_0\|_{\wh{\dB}_{p,\s}^s} 
      + \nu^{-1+\frac{1}{r_1}} \|f\|_{\wt{L^{r_1}(I};\wh{\dB}_{p,\s}^{s-2+\frac{2}{r_1}})} \right). 
\end{equation*} 
\end{lem}

\subsection{Culculus facts}
The following two inequalities play crucial role in the proof of decay estimates 
in Theorem \ref{thm;Lp-L1} (for the proofs, see e.g., \cite{Ch-Da}). 
\begin{lem} \label{lem;ce}
For any $a$, $b > 0$ with $\max (a,b)>1$, there exists a positive constant $C$ such that 
$$
\int_0^t
\langle t - \t \rangle^{-a} \langle \t \rangle^{-b}
d\t
\le C \langle t \rangle^{- \min(a,b)} \quad \text{for all } t \ge 0. 
$$
\end{lem}

\begin{lem}[{\it Page 73 in} \cite{B-C-D}] \label{lem;cineq}
Let $\delta_0>0$. For all $\s>0$, then there exists a constant $C_\s>0$ depending only on $\s$ such that 
$$
\sup_{t\ge 0}\sum_{j\in\Z}\big(t^\frac{1}{2}2^j\big)^\s e^{-t\delta_02^{2j}} \le C_\s. 
$$
\end{lem}

\sect{Linear analysis} \label{sect;linear}

In what follows, we assume $\r_*=1$ and $a:=\r-1$. We rapidly obtain the following equations for the 
perturbation denoted by $(a, m)$:  
\begin{equation} \label{eqn;mNSK2} 
\left\{
\begin{aligned}
& \pt_ta+\div m=0, 
&t>0, x \in \R^d, \\
&\pt_t m -\mathcal{L}m+\gm^2 \N a -\kappa \N \Delta a = N(a,m)
&t>0, x \in \R^d, \\
&(a, m)|_{t=0}=(a_0,m_0), &x \in \R^d,   
\end{aligned}
\right.
\end{equation}
where $\gm:=\sqrt{P'(1)}$ and the nonlinear part $N(a,m)$ is given by 
\begin{gather} 
N(a,m)=\text{div} \left((I(a)-1) m \otimes m\right)-I_P(a)\N a 
+ \mathcal{L}(I(a) m)+ \div \mathcal{K}(a), \label{eqn;nonlin} \\ 
I(a):=\frac{a}{1+a}, \quad 
I_P(a):=P'(1+a)-1.    \label{eqn;I}
\end{gather}

In this section, we present the pointwise behavior and regularizing estimates for the solution to the following linearized system 
at the Fourier side:  
\begin{equation} \label{eqn;L} 
\left\{
\begin{aligned}
& \pt_t \wh{a}+i \xi \cdot \wh{m}=\wh{f}, 
&t>0, x \in \R^d, \\
&\pt_t \wh{m} +\mu |\xi|^2 \wh{m}+(\lam+\mu)\xi (\xi \cdot \wh{m}) +i \gm^2 \xi \wh{a} +i \kappa \xi |\xi|^2 \wh{a} = \wh{g}
&t>0, x \in \R^d, \\
&(\wh{a}, \wh{m})|_{t=0}=(\wh{a}_0,\wh{m}_0), &x \in \R^d.   
\end{aligned}
\right.
\end{equation}

\begin{prop}[{\it Pointwise estimate in the Fourier space} \cite{Ch-Da-Xu,Ch-Ko,Na2}\hspace{0mm}] \label{lem;pw_L}
Let $(a,m)$ be a solution to the problem \eqref{eqn;L}. 
There exists some positive constant $C>0$ and $\delta_0>0$ 
such that   
\begin{align*}
\big|\big(\langle\xi\rangle\wh{a},\wh{m}\big)(t,\xi)\big| 
\le Ce^{-t\delta_0|\xi|^2}\big|\big(\langle\xi\rangle\wh{a},\wh{m}\big)(0,\xi)\big| 
+C\int_0^t e^{-(t-\t)\delta_0|\xi|^2}\big|\big(\langle\xi\rangle\wh{f},\wh{g}\big)\big|\,d\t. 
\end{align*}
\end{prop}

\begin{prop}[{\it Regularizing estimates for \eqref{eqn;L}} \cite{Na2}\hspace{0mm}] \label{prop;MR3}
Let $s \in \R$, $1 \le p,\s \le \infty$, $1 \le r_1 \le r \le \infty$ and $I=(0,T)$ with $T \in \R_+$. 
The inhomogeneous problem \eqref{eqn;L} has a unique solution $(\r,u,B)$ and 
there exists some constant $C >0$ such that 
\begin{equation*} 
\|(\langle\N\rangle a,m)\|_{\wt{L^r(I;}\wh{\dB}_{p,\s}^{s+\frac{2}{r}})}
\le C \left(\|\langle\N\rangle a_0,m_0)\|_{\wh{\dB}_{p,\s}^s} 
      +\|(\langle\N\rangle f,g)\|_{\wt{L^{r_1}(I;}\wh{\dB}_{p,\s}^{s-2+\frac{2}{r_1}})} \right). 
\end{equation*} 
\end{prop}

On the other hand, by solving the wave equation with strong damping obtained by the homogeneous problem, 
namely \eqref{eqn;L} with $f$, $g\equiv 0$ (see, e.g. Lemma 3.1 in \cite{Ho-Zu} or \S\,3 in \cite{Ko-Ts}), we see that 
the solution $(a,m)$ of the homogeneous problem can be written down explicitly as followings: 
\begin{align*}
&\wh{a}(t,\xi) 
=\frac{\lam_+(\xi)e^{\lam_-(\xi)t}-\lam_-(\xi)e^{\lam_+(\xi)t}}{\lam_+(\xi)-\lam_-(\xi)} \wh{a}_0(\xi)
  -i\,{}^t\xi \frac{e^{\lam_+(\xi)t}-e^{\lam_-(\xi)t}}{\lam_+(\xi)-\lam_-(\xi)} \wh{m}_0(\xi), \\
&\wh{m}(t,\xi)
= e^{-\mu|\xi|^2t}\wh{m}_0(\xi)-i\xi(\gm^2+\kappa|\xi|^2)\frac{e^{\lam_+(\xi)t}-e^{\lam_-(\xi)t}}{\lam_+(\xi)-\lam_-(\xi)}\wh{a}_0(\xi) \\
&\qquad \qquad \qquad \qquad \quad  +\left(\frac{\lam_+(\xi)e^{\lam_+(\xi)t}-\lam_-(\xi)e^{\lam_-(\xi)t}}{\lam_+(\xi)-\lam_-(\xi)}-e^{-\mu|\xi|^2t}\right) 
  \frac{\xi(\xi\cdot\wh{m}_0(\xi))}{|\xi|^2}, 
\end{align*}
where $\lam_\pm(\xi)=-\frac{\nu}{2}|\xi|^2(1 \pm \sqrt{1-\frac{4\kappa}{\nu^2}-\frac{4 \gm^2}{\nu^2|\xi|^2}})$ are the solution of the corresponding characteristic equation 
$\lam^2+\nu|\xi|^2\lam+|\xi|^2(\gm^2+\kappa |\xi|^2)=0$.  
Now, we set $\mathcal{G}^{i,j}=\mG^{i,j}(t,\xi)$ $(i,j=1,2)$ as followings: 
\begin{equation}  \label{eqn;mg}
\begin{aligned}
&\mG^{1,1}(t,\xi):=\frac{\lam_+(\xi)e^{\lam_-(\xi)t}-\lam_-(\xi)e^{\lam_+(\xi)t}}{\lam_+(\xi)-\lam_-(\xi)}, \quad 
\mG^{1,2}(t,\xi):=i\,{}^t\xi \frac{e^{\lam_+(\xi)t}-e^{\lam_-(\xi)t}}{\lam_+(\xi)-\lam_-(\xi)}, \\
&\mG^{2,1}(t,\xi):={-i\xi(\gm^2+\kappa|\xi|^2)\frac{e^{\lam_+(\xi)t}-e^{\lam_-(\xi)t}}{\lam_+(\xi)-\lam_-(\xi)}}, \\
&\mG^{2,2}(t,\xi):={\frac{\lam_+(\xi)e^{\lam_+(\xi)t}-\lam_-(\xi)e^{\lam_-(\xi)t}}{\lam_+(\xi)-\lam_-(\xi)} \cdot \frac{\xi{}^t\xi}{|\xi|^2}
                             +e^{-\mu|\xi|^2t}\left(\text{Id}-\frac{\xi{}^t\xi}{|\xi|^2}\right)},   
\end{aligned}
\end{equation}
where ${}^t \xi$ is the transpose of $\xi$. 
Hereafter, we also set $G^{i,j}(t,x):=\F^{-1}[\mG^{i,j}(t,\cdot)]$ $(i,j=1,2)$ and 
\begin{align*}
\mG(t,\xi)
:=
\left[
\begin{array}{@{\,}cc@{\,}}  
\mG^{1,1}(t,\xi) &\mG^{1,2}(t,\xi) \\
\mG^{2,1}(t,\xi) &\mG^{2,2}(t,\xi)
\end{array}
\right], \quad 
G(t,x):=
\left[
\begin{array}{@{\,}cc@{\,}}  
G^{1,1}(t,x) &G^{1,2}(t,x) \\
G^{2,1}(t,x) &G^{2,2}(t,x)
\end{array}
\right].  
\end{align*}
Let $U(t):={}^t(a,m)$ be solution to \eqref{eqn;mNSK2} with $U_0:={}^t(a_0,m_0)$. 
Since the solution of \eqref{eqn;L} is given by Green matrix $G$, we see that $U(t)$ is represented by 
\begin{equation} \label{eqn;IE}
\begin{aligned}
U(t)=&G(t,\cdot)*U_0
+\int_0^t 
G(t-\t,\cdot)* 
\left(
\begin{array}{@{\,}c@{\,}}
0\\
N(a,m)
\end{array}
\right)
d\t \\ 
=& 
\left(
\begin{array}{@{\,}c@{\,}}
G^{1,1}(t,\cdot)*a_0+G^{1,2}(t,\cdot)*m_0+\dsp \int_0^t G^{1,2}(t-\t,\cdot)*N(a,m)\,d\t\\
\dsp G^{2,1}(t,\cdot)*a_0+G^{2,2}(t,\cdot)*m_0+\int_0^t G^{2,2}(t-\t,\cdot)*N(a,m)\,d\t
\end{array}
\right). 
\end{aligned}
\end{equation}

The following lemma is shown by Kobayashi-Tsuda \cite{Ko-Ts}. The strategy of the proof is based on 
the {\it stationary phase methods} on the sphere $\mathbb{S}^{d-1}$ and the representation formula of the wave equation as in Shibata \cite{Sh} and 
Kobayashi-Shibata \cite{Ko-Sh}: 
\begin{lem}[The pointwise estimate in the low frequencies \cite{Ko-Ts}] \label{lem;DW}
Let $k=0,1,2,\cdots$,  $\al \in (\mathbb{N} \cup \{0\})^d$ 
and $\psi \in C^\infty(\mathbb{S}^{d-1})$.  The function $K_{\psi, L}(t,\xi)$ is defined by 
\begin{equation} \label{eqn;kpsi}
K_{\psi,L}(t,x)
:=\F^{-1}\left[\chi_L(\xi) \wh{K_\psi(t,\cdot)} \right] \;\; with \;\; 
K_{\psi}(t,x):=\F^{-1}\left[\frac{e^{\lam_+(\xi)t}-e^{\lam_-(\xi)t}}{\lam_+(\xi)-\lam_-(\xi)}\psi(\xi)\right]. 
\end{equation}
Then there exists a positive constant $C$ such that for all $t>0$, 
\begin{align*}
\|\pt_t^k \pt_x^\al K_{\psi,L}(t,\cdot)\|_{L^\infty}
\le C(1+t)^{-(\frac{3d-3}{4}+\frac{k+|\al|}{2})} 
\end{align*}
(for the definition of the low-frequency cut-off $\chi_L(\xi)$, see \S\,3 in Kobayashi-Tsuda \cite{Ko-Ts}). 
\end{lem}
%
%

\sect{Global well-posedness in the critical Fourier--Besov spaces} \label{sect;GWP}
In this section, we give the proof of Theorem \ref{thm;GWP_FLp}.  
In order to prove Theorem \ref{thm;GWP_FLp}, we apply Banach's fixed point argument. 
Let us define the complete metric space $CL_T$ and the solution mapping $\Phi=\Phi[(b,n)]$ as follows:  
Given initial data such that $(a_0,m_0)$ satisfy the assumptions as in Theorem \ref{thm;GWP_FLp}, 
the solution mapping $\Phi$ is defined by 
\begin{equation}
\Phi: (b,n) \longmapsto (a,m)
\end{equation} 
with $(a,m)$ the solution to 
\begin{equation} \label{eqn;AE} 
\left\{
\begin{aligned}
&\;\pt_t a+\div m=0, \\
&\;\pt_t m-\L m +\gm^2 \N a-\kappa\N \Del a =N(b,n), \\
&\;(a, m)|_{t=0}=(a_0,m_0), 
\end{aligned}
\right.
\end{equation}
where $N(b,n)=N(a,m)|_{(a,m)=(b,n)}$ are defined in \eqref{eqn;nonlin}. 

On the other hand, let us introduce the following Banach space $CL_T$ for any $T>0$: 
$$
CL_T:=\wt{\,L^\infty(0},T;\cfB^{-1+\frac{d}{p}})
      \cap L^1(0,T;\cfB^{1+\frac{d}{p}})
$$
And, let the above space $CL_T$ be equipped with the norm 
$$
\|(a,m)\|_{CL_T}
:=\|(\langle\N\rangle a,m)\|_{\wt{L^\infty_T(}\cfB^{-1+\frac{d}{p}})}
+\||\N|^2(\langle\N\rangle a,m)\|_{L^1_T(\cfB^{-1+\frac{d}{p}})}.  
$$

\begin{lem} \label{lem;nonlin_est}
Let $1 \le p \le \infty$ and $T>0$. Let $N$ be defined in \eqref{eqn;nonlin} and 
$P$ be a real analytic function in a neighborhood of $1$ such that $P'(1)>0$. 
There exists some positive constant $C>0$ such that 
\begin{equation} \label{est;nonlin_all}
\begin{aligned}
\|N(b,n)\|_{L^1_t(\cfB^{-1+\frac{d}{p}})}
\le C(\|(b,n)\|_{CL_T}^2+\|(b,n)\|_{CL_T}^3)
\end{aligned}
\end{equation}
for all $(b,n) \in CL_T$ and $0 \le t \le T$. 
\end{lem}

\begin{proof}[The proof of Lemma \ref{lem;nonlin_est}] 
Firstly, we notice that it follows from the interpolation inequality that 
\begin{align} \label{est;interpo}
 \wt{\,L^2(0},T;\cfB^{\frac{d}{p}}) 
 \hr \wt{\,L^\infty(0},T;\cfB^{-1+\frac{d}{p}}) 
     \cap L^1(0,T;\cfB^{1+\frac{d}{p}})
. 
\end{align}
\noindent
{\it The estimate for $\div ((I(b)-1)\,n \otimes n)$}: By Lemma \ref{lem;Bern} and \eqref{est;Banach_ring}, 
we see 
\begin{align} \label{est;nonlin1}
\|\div ((I(b)-1)\,n \otimes n)\|_{L^1_t(\cfB^{-1+\frac{d}{p}})}
\les \|(I(b)-1)\,n\otimes n\|_{L^1_t(\cfB^{\frac{d}{p}})}.  
\end{align}  
Thanks to the Taylor expansion, we see $(1+b)^{-1}=\sum_{n=0}^\infty (-1)^n b^n$ if $|b|<1$. 
By using \eqref{est;Banach_ring} and $\|b\|_{\wt{L^\infty_t(}\cfB^\frac{d}{p})} \le \ve_0 \ll 1$, 
we see that 
\begin{equation} 
\begin{aligned} \label{est;1/1+r}
\left\|\frac{b}{1+b}\right\|_{\wt{L^\infty_t(}\cfB^\frac{d}{p})} 
\le& \|b\|_{\wt{L^\infty_t(}\cfB^\frac{d}{p})} \sum_{n=0}^\infty \bigg(C\|b\|_{\wt{L^\infty_t(}\cfB^\frac{d}{p})}\bigg)^n \\
\le& \|b\|_{\wt{L^\infty_t(}\cfB^\frac{d}{p})} \sum_{n=0}^\infty \left(C\ve_0\right)^n
=\frac{1}{1-C\ve_0} \|b\|_{\wt{L^\infty_t(}\cfB^\frac{d}{p})}.  
\end{aligned}
\end{equation}
Combination of \eqref{est;nonlin1} and \eqref{est;1/1+r} gives us that 
\begin{align*}
\|\div ((I(b)-1)\,n \otimes n)\|_{L^1_t(\cfB^{-1+\frac{d}{p}})}
\les& \left(1+\left\|\frac{b}{1+b}\right\|_{L^\infty_t(\cfB^\frac{d}{p})}\right)
      \|n\otimes n\|_{L^1_t(\cfB^{\frac{d}{p}})} \\
\les& 
(1+\|b\|_{L^\infty_t(\cfB^\frac{d}{p})})\|n\|_{L^2_t(\cfB^{\frac{d}{p}})}^2 \\
\les& \|(b,n)\|_{CL_T}^2+\|(b,n)\|_{CL_T}^3. 
\end{align*}

\noindent
{\it The estimate for $I_P(b)\N b$}: 
Since $P$ is real analytic in a neighborhood of $1$, there exists some constant $R_P>0$ such that 
if $|b| \le R_P$ then 
$$
  P(b+1)=\sum_{n=0}^\infty a_n b^n \;\;\text{with}\;\;a_n=\frac{P^{(n)}(1)}{n!}. 
$$
Noting that $I_P(b)\N b=\N (b^2 \wt{I}_P(b))$ with $\wt{I}_P(b)=\sum_{n=2}^\infty a_n b^{n-2}$ 
and using \eqref{est;Banach_ring} and \eqref{est;1/1+r}, one can see that the following estimate for 
$\wt{I}_P(b)$ holds true:   
\begin{equation} \label{est;tlI_P}
\begin{aligned}
\|\wt{I}_P(b) b\|_{\wt{L^\infty_t(}\cfB^\frac{d}{p})}
&=\left\|\sum_{n=2}^\infty a_n b^{n-1}\right\|_{\wt{L^\infty_t(}\cfB^\frac{d}{p})} 
\les \bar{P}\left(C\|b\|_{\wt{L^\infty_t(}\cfB^\frac{d}{p})}\right) \|b\|_{\wt{L^\infty_t(}\cfB^\frac{d}{p})}, 
\end{aligned}
\end{equation}
where $\bar{P}(z):=\sum_{n=2}^\infty |a_n|z^n$. If $\|b\|_{\wt{L^\infty_t(}\cfB^\frac{d}{p})} \le \frac{R_P}{2C}$, 
it follows that  
\begin{align} \label{est;I_P}
\|\wt{I}_P(b) b\|_{\wt{L^\infty_t(}\cfB^\frac{d}{p})} \le D \|b\|_{\wt{L^r_t(}\cfB^\frac{d}{p})}
\;\;\text{with}\;\;D=1+\sup_{|z|\le\frac{R_P}{2}}|\bar{P}(z)|.
\end{align}
For any $b \in L^r_t(\cfB^{d/p}) \cap L^\infty_t(\cfB^{d/p})$, 
the same estimate holds in Bochner space $L^r_t(\cfB^{d/p})$: 
\begin{align} \label{est;I_P2}
\|\wt{I}_P(b)b\|_{L^r_t(\cfB^\frac{d}{p})} \les D\|b\|_{L^r_t(\cfB^\frac{d}{p})}
\;\;\text{with}\;\;D=1+\sup_{|z|\le\frac{R_P}{2}}|\bar{P}(z)|.
\end{align}
By \eqref{est;Banach_ring} and \eqref{est;I_P2}, we immediately obtain that 
\begin{align*}
\|\N(b^2 \wt{I}_P(b))\|_{L^1_t(\cfB^{-1+\frac{d}{p}})}
\les \|b^2 \wt{I}_P(b)\|_{L^1_t(\cfB^{\frac{d}{p}})}
\les \|\wt{I}_P(b)b\|_{L^\infty_t(\cfB^\frac{d}{p})}\|b\|_{L^2_t(\cfB^{\frac{d}{p}})}
\les \|(b,n)\|_{CL_T}^2. 
\end{align*}
\noindent
{\it The estimate for $\mathcal{L}(I(b)n)$}: Since 
$\Delta (I(b)n)=\div (\N I(b)\otimes n+I(b)\N n)$, it holds that 
\begin{align*}
\|\Delta (I(b)n)\|_{L^1_t(\cfB^{-1+\frac{d}{p}})}
\les& \|\N I(b)\otimes n\|_{L^1_t(\cfB^{\frac{d}{p}})}+\|I(b) \N n\|_{L^1_t(\cfB^{\frac{d}{p}})} \\
\les& \|\N I(b)\|_{L^2_t(\cfB^{\frac{d}{p}})}\|n\|_{L^2_t(\cfB^{\frac{d}{p}})}
      +\|I(b)\|_{L^\infty_t(\cfB^{\frac{d}{p}})} \|\N n\|_{L^1_t(\cfB^{\frac{d}{p}})}. 
\end{align*}
For $\N I(b)$, it follows from Lemma \ref{lem;bil} with $\wh{L}^\infty(\R^d) \hr \cfB^{d/p}(\R^d)$, \eqref{est;Banach_ring} that 
\begin{equation} \label{est;Na}
\begin{aligned}
\|\N I(b)\|_{L^2_t(\cfB^{\frac{d}{p}})}
=& \left\|\sum_{n=0}^\infty (-1)^n b^{n+1}\right\|_{L^2_t(\cfB^{1+\frac{d}{p}})} \\
\les&  \|b\|_{L^2_t(\cfB^{1+\frac{d}{p}})} \sum_{n=0}^\infty \bigg(C\|b\|_{L^\infty_t(\cfB^\frac{d}{p})}\bigg)^n
\les \|\N b\|_{L^2_t(\cfB^{\frac{d}{p}})},  
\end{aligned}
\end{equation}
which yields to 
\begin{align*}
\|\Delta (I(b)n)\|_{L^1_t(\cfB^{-1+\frac{d}{p}})}
\les& \|\N b\|_{L^2_t(\cfB^{\frac{d}{p}})}\|n\|_{L^2_t(\cfB^{\frac{d}{p}})} \\
      &+\|b\|_{L^\infty_t(\cfB^{\frac{d}{p}})}\|\N n\|_{L^1_t(\cfB^{\frac{d}{p}})} 
\les \|(b,n)\|_{CL_T}^2. 
\end{align*}
Performing the same calculation as above estimate, one can obtain that 
\begin{align*}
\|\N \div (I(b)n)\|_{L^1_t(\cfB^{-1+\frac{d}{p}})} \les \|(b,n)\|_{CL_T}^2. 
\end{align*}

\noindent
{\it The estimate for $\div \mathcal{K}(b)$}: 
By the definition of $\mathcal{K}(b)$ in \eqref{eqn;Korteweg}, 
\eqref{est;Banach_ring} and \eqref{est;1/1+r}, we have 
\begin{align*}
\|\div \mathcal{K}(b)\|_{L^1_t(\cfB^{-1+\frac{d}{p}})}
\les& \|\Delta b^2\|_{L^1_t(\cfB^{\frac{d}{p}})}+\||\N b|^2\|_{L^1_t(\cfB^{\frac{d}{p}})}
      +\|\N b \otimes \N b\|_{L^1_t(\cfB^{\frac{d}{p}})} \\
\les& \|b\|_{L^\infty_t(\cfB^\frac{d}{p})}\|b\|_{L^1_t(\cfB^{2+\frac{d}{p}})}
+\|b\|^2_{L^2_t(\cfB^{1+\frac{d}{p}})} \\
\les& \|(b,n)\|_{CL_T}^2.   
\end{align*}
Gathering all, we thus obtain the desired estimate \eqref{est;nonlin_all}. 
\end{proof}

\begin{proof}[The proof of Theorem \ref{thm;GWP_FLp}]  
The proof readily follows from Propositon \ref{prop;MR3}, Lemma \ref{lem;nonlin_est} and the following auxiliary lemma 
(for the proof, see e.g., \cite{Ch-Ko}):
\begin{lem}[\cite{Ch-Ko}] \label{lem;Banach_fixed_pt}
Let $(X,\|\cdot\|_X)$ be a Banach space. 
Let $B:X \times X \to X$ be a bilinear continuous operator with norm $K_2$ and 
$T:X \times X \times X \to X$ be a trilinear operators with norm $K_3$. 
Let further $L:X \to X$ be a continuous linear operator with norm $N<1$. 
Then for all $y \in X$ such that 
$$
\|y\|_X <\min \left(\frac{1-N}{2},\frac{(1-N)^2}{2(2K_2+3K_3)}\right), 
$$
the equation $x=y+L(x)+B(x,x)+T(x,x,x)$ has a unique solution $x$ in the ball $B_{\tilde{R}}^X(0)$ of center $0$ and 
radius $\tilde{R}=\min (1,\frac{1-N}{2K_2+3K_3})$. In addition, $x$ satisfies 
$$
\|x\|_X \le \frac{2}{1-N}\|y\|_X. 
$$
\end{lem}
By Proposition \ref{prop;MR3} and Lemma \ref{lem;nonlin_est}, we see that $(a,m)$ satisfies 
\begin{equation} \label{est;pre_unif}
\begin{aligned}
\|(a,m)\|_{CL_T} 
\les& \|(\langle\N\rangle a_0,m_0)\|_{\cfB^{-1+\frac{d}{p}}}
      +\|N(b,n)\|_{L^1_t(\cfB^{-1+\frac{d}{p}})}.  
\end{aligned}
\end{equation}
Lemma \ref{lem;nonlin_est} ensures that $N(b,n)$ can be regarded as a combination of 
bi-and-trilinear continuous operators in $CL_T$ for any $T>0$. We define a ball in $CL_T$ centered at the origin by 
$$
B_R^{CL_T}(0):=\{(a,m) \in CL_T;\|(a,m)\|_{CL_T} \le R\}, 
$$ 
where $R>0$. Lemma \ref{lem;Banach_fixed_pt} shows that the existence of a unique solution in $B_R^{CL_T}(0)$ for a sufficiently small data. 
Moreover, we are able to take $T=\infty$, thanks to the uniform estimate with respect to $t$ in Lemma \ref{lem;nonlin_est}. 
Therefore, we obtain a global-in-time solution satisfying \eqref{est;unif}. 
The time-consinuity of $(a,m)$ follows from Proposition \ref{prop;MR3} and the density 
of $\mathcal{S}_0$ in $\cfB^s$, where $\mathcal{S}_0$ is the set of functions in $\mathcal{S}$ whose Fourier transforms are supported away from $0$. 
\end{proof}

\sect{The $L^p$-$L^1$ type decay estimates for global solutions with $1 \le p \le \frac{2d}{d-1}$} 
\label{sect;Lp-L1}

The following estimate is immediately derived from Proposition \ref{prop;MR3}, 
Lemma \ref{lem;nonlin_est} and standard argument for the cubic equation.  
We would like to state the assertion without its proof.  
\begin{prop}[{\it Uniform estimate}\hspace{0mm}] \label{prop;unif}
  Let $(a_0,m_0)$ satisfy the assumption of Theorem \ref{thm;GWP_FLp} 
  and $(a,m)$ be a solution to the problem \eqref{eqn;mNSK2}. 
  There exists a threshold number $j_0 \in \Z$ and some positive constant $C$ such that 
$$
  \mathcal{X}_p(t) \le C\mathcal{X}_{p,0} \quad for \;any \;\; t>0, 
$$ 
where $\mathcal{X}_p(t)$ and $\mathcal{X}_{p,0}$ are defined in Theorem \ref{thm;GWP_FLp}. 
\end{prop}

In this section, we give the proof of Theorem \ref{thm;Lp-L1}. Firstly, 
let us introduce the following functional concerning the time-decay of global solutions for any $t>0$:  
\begin{equation}
\begin{aligned}
\mD_p(t)
:=&\sup_{s \in (-\frac{d}{p'},1+\frac{d}{p}]}
   \|\langle\t\rangle^{\frac{d}{2p'}+\frac{s}{2}}(a,m)\|_{L^\infty(I;\cfB^s)}^\ell \\
   &+\|\langle\t\rangle^\al (\N a,m)\|_{\wt{L^\infty(I};\cfB^{-1+\frac{d}{p}})}^h 
    +\|\t^\al(\N a,m)\|_{\wt{L^\infty(I};\cfB^{1+\frac{d}{p}})}^h, 
\end{aligned}
\end{equation}
where $\al:=d-\ve$ with sufficiently small constant $\ve>0$, $I:=(0,t)$.  
By Lemma \ref{lem;pw_L}, we immediately see that the global solution $(a,m)$ to \eqref{eqn;mNSK2} satisfies the 
following estimate: 
\begin{equation} \label{est;F-L}
\begin{aligned}
\big\|\big(\langle\xi\rangle\wh{a}_j,\wh{m}_j\big)(t,\xi)\big\|_{L^{p'}} 
\les& e^{-t\frac{\delta_0}{4}2^{2j}}
     \big\|\wh{\phi}_j\big(\langle\xi\rangle\wh{a}_0,\wh{m}_0\big)\big\|_{L^{p'}} \\
&+\int_0^t
     e^{-(t-s)\frac{\delta_0}{4}2^{2j}}
     \big\|\wh{\phi}_j \wh{N(a,m)}\big\|_{L^{p'}}
    ds, 
\end{aligned}
\end{equation}
where $N(a,m)$ is already defined in \eqref{eqn;nonlin}. 

\subsection{The time-weighted estimates for the low-frequency part}
Let us consider the boundedness for the first term of $\mD_p(t)$.   
For any $s>-d/{p'}$, we obtain by $|\xi| \les 2^{j_0}$ if $\xi \in \supp \wh{\phi}_{j_0}$ and Lemma \ref{lem;cineq} that 
\begin{equation} \label{est;heat}
\begin{aligned}
\langle t\rangle^{\frac{d}{2p'}+\frac{s}{2}}
&\sum_{j \le j_0}2^{sj}e^{-t\frac{\dl_0}{4}2^{2j}}\|\wh{\phi}_j(\langle\xi\rangle\wh{a}_0,\wh{m}_0)\|_{L^{p'}} \\
\les& \sum_{j \le j_0}
      (1+t^{\frac{d}{2p'}+\frac{s}{2}})2^{(s+\frac{d}{p'})j}
      2^{-\frac{d}{p'}j}e^{-tc_02^{2j}}\|\wh{\phi}_j(\wh{a}_0,\wh{m}_0)\|_{L^{p'}}\\
\les& \|(a_0,m_0)\|_{\fB_{p,\infty}^{-\frac{d}{p'}}}^\ell
      \left(1+\sum_{j\le j_0}(t^{\frac{1}{2}}2^j)^{s+\frac{d}{p'}}e^{-tc_02^{2j}}\right)
\les \mD_{p,0}. 
\end{aligned}
\end{equation}
Here we set $c_0=\dl_0/4$.  
Therefore, we see that for any $s \in (-d/p',1+d/p]$, 
\begin{equation} \label{decay;low_ruB}
\begin{aligned}
\|(a,m)(t)\|_{\cfB^s}^\ell
\les \langle t\rangle^{-(\frac{d}{2p'}+\frac{s}{2})} \mD_{p,0}
    +\int_0^t\langle t-\t\rangle^{-(\frac{d}{2p'}+\frac{s}{2})}\|N(a,m)(\t)\|_{\fB_{p,\infty}^{-\frac{d}{p'}}}^\ell d\t. 
\end{aligned}
\end{equation}

In what follows, let us consider the estimate for the time-integral part. 
In order to complete such an estimation, we would like to introduce the following lemma: 
\begin{lem}[Lemma 5.2 in \cite{Na2}] \label{lem;cA-P}
Let $d \ge 1$, $s_i \in \R$, $1 \le p,p_i \le \infty$ $(i=1,2)$ satisfying 
$$
\frac{1}{p}\le\frac{1}{p_1}+\frac{1}{p_2}, \quad s_1 \le \frac{d}{p_1}, \quad s_2 <\frac{d}{p_2}, \quad 
s_1+s_2+d\min\left(0,1-\frac{1}{p_1}-\frac{1}{p_2}\right) \ge 0. 
$$
Then there exists some constant $C>0$ such that 
\begin{align} \label{est;prod_diag}
\left\|fg\right\|_{\fB_{p,\infty}^{s_1+s_2+\frac{d}{p}-\frac{d}{p_1}-\frac{d}{p_2}}}
\le C\|f\|_{\fB_{p_1,1}^{s_1}}\|g\|_{\fB_{p_2,\infty}^{s_2}}. 
\end{align}
\end{lem}
For all $1 \le p \le 2d/(d-1)$, Lemma \ref{lem;cA-P} gives us the following estimates:    
\begin{align}
&\|fg\|_{\fB_{p,\infty}^{1-\frac{d}{p'}}}
\les \|f\|_{\cfB^\frac{d}{p}}\|g\|_{\fB^{1-\frac{d}{p'}}_{p,\infty}} \label{est;p-1}\\
&\|fg\|_{\fB_{p,\infty}^{1-\frac{d}{p'}}}
\les \|f\|_{\cfB^{-1+\frac{d}{p}}}\|g\|_{\fB_{p,\infty}^{2-\frac{d}{p'}}}. \label{est;p-2}
\end{align}

\begin{prop}
\label{prop;decay_low_N} 
Let $s \in \R$ with $-d/p'<s\le 1+d/p$ and $1 \le p \le 2d/(d-1)$. 
Then there exists some constant $C>0$ such that for any $t > 0$, 
\begin{align*}
\int_0^t\langle t-\t\rangle^{-(\frac{d}{2p'}+\frac{s}{2})}\||\N|^{-1}N(a,m)(s)\|_{\fB_{p,\infty}^{1-\frac{d}{p'}}}^{\ell} d\t
\le C\langle t\rangle^{-(\frac{d}{2p'}+\frac{s}{2})}(1+\mX_p(t))(\mX_p(t)^2+\mD_p(t)^2). 
\end{align*}
\end{prop}

\begin{proof}[The proof of Proposition \ref{prop;decay_low_N}]
By \eqref{est;1/1+r}, \eqref{est;p-1}, \eqref{est;p-2}, Lemma \ref{lem;ce} and the embedding 
$\cfB^s(\R^d) \hr \fB_{p,\infty}^s(\R^d)$, we obtain that for any $1 \le p \le 2d/(d-1)$ and $t>0$,  
\begin{align*}
\int_0^t\langle t-\t\rangle^{-(\frac{d}{2p'}+\frac{s}{2})}&
  \|(I(a)-1)m \otimes m^\ell\|_{\fB_{p,\infty}^{1-\frac{d}{p'}}}^\ell d\t \\
&\les \int_0^t
        \langle t-\t\rangle^{-(\frac{d}{2p'}+\frac{s}{2})}
        \Big(1+\|a\|_{\cfB^\frac{d}{p}}\Big) \|m\|_{\cfB^{-1+\frac{d}{p}}} \|m^\ell\|_{\cfB^{2-\frac{d}{p'}}}d\t \\
&\les (1+\mX_p(t)) \int_0^t
        \langle t-\t\rangle^{-(\frac{d}{2p'}+\frac{s}{2})}
        \Big(\|m\|_{\cfB^{-1+\frac{d}{p}}}^\ell+\|m\|_{\cfB^{-1+\frac{d}{p}}}^h\Big) 
        \|m\|_{\cfB^{2-\frac{d}{p'}}}^\ell d\t\\
&\les (1+\mX_p(t))\mD_p(t)^2
      \int_0^t
        \langle t-\t\rangle^{-(\frac{d}{2p'}+\frac{s}{2})}
        \langle\t\rangle^{-\min(\frac{d}{2}-\frac{1}{2},\al)}\langle \t \rangle^{-1}
      d\t \\ 
&\les \langle t\rangle^{-(\frac{d}{2p'}+\frac{s}{2})}(1+\mX_p(t)) \mD_p(t)^2. 
\end{align*}
For the high-frequency part $m \otimes m^h$, we consider the cases $t \le 2$ and $t>2$, separately.  
In the case $t \le 2$, we obtain by \eqref{est;1/1+r}, \eqref{est;p-1}, \eqref{est;p-2}, Proposition \ref{prop;unif} and noting that $\langle t\rangle \les 1$ 
and $2-d/p' \le 1+d/p$ for $d \ge 2$, 
\begin{align*}
\int_0^t\langle t-\t\rangle^{-(\frac{d}{2p'}+\frac{s}{2})}\|(I(a)-1) m \otimes m^h\|_{\fB_{p,\infty}^{1-\frac{d}{p'}}}^\ell d\t
\les& (1+\mX_p(t)) \int_0^t \|m\|_{\cfB^{-1+\frac{d}{p}}}\|m^h\|_{\cfB^{2-\frac{d}{p'}}}d\t \\
\les& \langle t\rangle^{-(\frac{d}{2p'}+\frac{s}{2})}\langle t\rangle^{\frac{d}{2p'}+\frac{s}{2}} (1+\mX_p(t)) \mX_p(t)^2 \\
\les& \langle t\rangle^{-(\frac{d}{2p'}+\frac{s}{2})} (1+\mX_p(t)) \mX_p(t)^2. 
\end{align*}
On the other hand, in the case $t>2$, we decompose as 
\begin{align*}
\int_0^t\langle t-\t\rangle^{-(\frac{d}{2p'}+\frac{s}{2})}&
\|(I(a)-1)m\otimes m^h\|_{\fB_{p,\infty}^{1-\frac{d}{p'}}}^\ell d\t \\
&\les (1+\mX_p(t)) \int_0^t\langle t-\t\rangle^{-(\frac{d}{2p'}+\frac{s}{2})}
\|m\otimes m^h\|_{\fB_{p,\infty}^{1-\frac{d}{p'}}}^\ell d\t \\
&\les (1+\mX_p(t)) \left(\int_0^1+\int_1^t\right)
  \langle t-\t\rangle^{-(\frac{d}{2p'}+\frac{s}{2})}\|m\otimes m^h\|_{\fB_{p,\infty}^{1-\frac{d}{p'}}}^\ell d\t. 
\end{align*}
Notice that $\langle t\rangle \langle t-\t\rangle \les 1$ if $0\le\t\le1$, it follows from \eqref{est;p-2} that 
\begin{align*}
  \int_0^1&
  \langle t-\t\rangle^{-(\frac{d}{2p'}+\frac{s}{2})}\|m\otimes m^h\|_{\fB_{p,\infty}^{1-\frac{d}{p'}}}^\ell 
 d\t \\ 
 &\les \langle t\rangle^{-(\frac{d}{2p'}+\frac{s}{2})}
     \int_0^1 
      \left(\frac{\langle t\rangle}{\langle t-\t\rangle}\right)^{\frac{d}{2p'}+\frac{s}{2}}
      \|m\|_{\cfB^{-1+\frac{d}{p}}}\|m^h\|_{\cfB^{2-\frac{d}{p'}}}
     d\t
\les \langle t\rangle^{-(\frac{d}{2p'}+\frac{s}{2})}\mX_p(1)^2. 
\end{align*}
Since $\langle\t\rangle\t^{-1} \les 1$ if $\t >1$, we obtain by \eqref{est;p-2} and Lemma \ref{lem;ce} that 
\begin{align*}
  \int_1^t
  \langle t-\t\rangle^{-(\frac{d}{2p'}+\frac{s}{2})} & \|m\otimes m^h\|_{\fB_{p,\infty}^{1-\frac{d}{p'}}}^\ell 
 d\t \\
 &\les \int_1^t
            \langle t-\t\rangle^{-(\frac{d}{2p'}+\frac{s}{2})}
            \|m\|_{\cfB^{-1+\frac{d}{p}}}\|m^h\|_{\cfB^{2-\frac{d}{p'}}}
           d\t \\
&\les \int_1^t
        \langle t-\t\rangle^{-(\frac{d}{2p'}+\frac{s}{2})}
        \Big(\|m\|_{\cfB^{-1+\frac{d}{p}}}^\ell+\|m\|_{\cfB^{-1+\frac{d}{p}}}^h\Big)
        \|m\|_{\cfB^{1+\frac{d}{p}}}^h d\t \\
&\les \mD_p(t)^2
      \int_1^t
      \langle t-\t\rangle^{-(\frac{d}{2p'}+\frac{s}{2})}
      \langle\t\rangle^{-\min(\frac{d}{2}-\frac{1}{2}, \al)}
      \langle\t\rangle^{-\al}\cdot \frac{\langle\t\rangle^\al}{\t^\al}
      d\t \\ 
&\les \langle t\rangle^{-(\frac{d}{2p'}+\frac{s}{2})}\mD_p(t)^2. 
\end{align*}
Combining the above estimates for the low-and high-frequency part, we thus obtain 
\begin{equation}
\begin{aligned} \label{decay;N_1}
\int_0^t\langle t-\t\rangle^{-(\frac{d}{2p'}+\frac{s}{2})}&\|(I(a)-1)m \otimes m\|_{\fB_{p,\infty}^{1-\frac{d}{p'}}}^\ell d\t \\
&\les \langle t\rangle^{-(\frac{d}{2p'}+\frac{s}{2})}(1+\mX_p(t)) (\mX_p(t)^2+\mD_p(t)^2). 
\end{aligned}
\end{equation}

Since $I_P(a)\N a= \N(a^2 \wt{I}_P(a))$, 
we have by applying \eqref{est;I_P}, \eqref{est;p-1} and Lemma \ref{lem;ce} to the estimate for 
$\N(a^2\wt{I}_P(a))$ that 
\begin{equation} \label{decay;N_2}
\begin{aligned}
\int_0^t 
   \langle t-\t&\rangle^{-(\frac{d}{2p'}+\frac{s}{2})}
  \|a^2 \wt{I}_P(a)\|_{\fB_{p,\infty}^{1-\frac{d}{p'}}}^\ell d\t \\
  &\les \int_0^t 
       \langle t-\t\rangle^{-(\frac{d}{2p'}+\frac{s}{2})}
       \Big(\|a\|_{\cfB^\frac{d}{p}}^\ell+\|a\|_{\cfB^\frac{d}{p}}^h\Big)
       \Big(\|a\|_{\fB_{p,\infty}^{1-\frac{d}{p'}}}^\ell+\|a\|_{\fB_{p,\infty}^{1-\frac{d}{p'}}}^h\Big) 
      d\t \\
&\les \mD_p(t)^2
      \int_0^t
       \langle t-\t\rangle^{-(\frac{d}{2p'}+\frac{s}{2})}
       \langle\t\rangle^{-\min(\frac{d}{2},\al)}\langle\t\rangle^{-\min(\frac{1}{2},\al)} 
      d\t \\
&\les \langle t\rangle^{-(\frac{d}{2p'}+\frac{s}{2})} \mD_p(t)^2. 
\end{aligned}
\end{equation}

In a similar way to the estimate for $(I(a)-1)m \otimes m^\ell$,  
we obtain by using \eqref{est;1/1+r}, \eqref{est;p-1} and Lemma \ref{lem;ce} that for any $1 \le p \le 2d/(d-1)$ and $t>0$, 
\begin{align*}
\int_0^t\langle t-\t\rangle^{-(\frac{d}{2p'}+\frac{s}{2})}
  \||\N|^{-1}\Delta (I(a) m^\ell)\|_{\fB_{p,\infty}^{1-\frac{d}{p'}}}^\ell d\t 
\les& \int_0^t
        \langle t-\t\rangle^{-(\frac{d}{2p'}+\frac{s}{2})}
        \|I(a) m^\ell\|_{\fB_{p,\infty}^{1-\frac{d}{p'}}}^\ell d\t \\
\les& \int_0^t
        \langle t-\t\rangle^{-(\frac{d}{2p'}+\frac{s}{2})}
        \|a\|_{\cfB^{\frac{d}{p}}} \|m^\ell\|_{\cfB^{1-\frac{d}{p'}}}d\t \\
\les& \mD_p(t)^2
      \int_0^t
        \langle t-\t\rangle^{-(\frac{d}{2p'}+\frac{s}{2})}
        \langle\t\rangle^{-\min(\frac{d}{2},\al)}\langle\t\rangle^{-\frac{1}{2}}
      d\t \\
\les& \langle t\rangle^{-(\frac{d}{2p'}+\frac{s}{2})}\mD_p(t)^2. 
\end{align*}
For the high-frequency part of $\Delta(I(a)m^h)$, in the case $t \le 2$, we see that 
\begin{align*}
\int_0^t\langle t-\t\rangle^{-(\frac{d}{2p'}+\frac{s}{2})}
  \|I(a) m^h\|_{\fB_{p,\infty}^{1-\frac{d}{p'}}}^\ell d\t 
\les& \int_0^t\langle t-\t\rangle^{-(\frac{d}{2p'}+\frac{s}{2})}
       \|a\|_{\cfB^\frac{d}{p}}\|m\|_{\cfB^{1-\frac{d}{p'}}}^hd\t \\
\les& \langle t\rangle^{-(\frac{d}{2p'}+\frac{s}{2})}\mX_p(t)^2. 
\end{align*} 
On the other hand, we see that if $t>2$,  
\begin{align*}
\int_0^t\langle t-\t\rangle^{-(\frac{d}{2p'}+\frac{s}{2})}
  \|I(a) m^h\|_{\fB_{p,\infty}^{1-\frac{d}{p'}}}^\ell d\t 
\les& \left(\int_0^1+\int_1^t\right)\langle t-\t\rangle^{-(\frac{d}{2p'}+\frac{s}{2})}
       \|a\|_{\cfB^\frac{d}{p}}\|m\|_{\cfB^{1-\frac{d}{p'}}}^hd\t \\
\les& \langle t\rangle^{-(\frac{d}{2p'}+\frac{s}{2})}\mX_p(1)^2\\
     &+\int_1^t 
        \langle t-\t\rangle^{-(\frac{d}{2p'}+\frac{s}{2})}
        \langle\t\rangle^{-\min(\frac{d}{2},\al)}
        \langle\t\rangle^{-\al}\cdot \frac{\langle\t\rangle^\al}{\t^\al} 
      d\t \\
\les& \langle t\rangle^{-(\frac{d}{2p'}+\frac{s}{2})}(\mX_p(t)^2+\mD_p(t)^2). 
\end{align*}
Gathering all, we thus obtain 
\begin{equation} \label{decay;N_3}
\begin{aligned}
\int_0^t\langle t-\t\rangle^{-(\frac{d}{2p'}+\frac{s}{2})}
 \||\N|^{-1}\Delta (I(a) m)\|_{\fB_{p,\infty}^{1-\frac{d}{p'}}}^\ell d\t
\les \langle t\rangle^{-(\frac{d}{2p'}+\frac{s}{2})}(\mX_p(t)^2+\mD_p(t)^2). 
\end{aligned}
\end{equation}
Analogously, it follows that 
\begin{equation} \label{decay;N_4}
\begin{aligned}
\int_0^t\langle t-\t\rangle^{-(\frac{d}{2p'}+\frac{s}{2})}
 \||\N|^{-1} \N \div (I(a) m)\|_{\fB_{p,\infty}^{1-\frac{d}{p'}}}^{\ell} d\t
\les \langle t\rangle^{-(\frac{d}{2p'}+\frac{s}{2})}(\mX_p(t)^2+\mD_p(t)^2). 
\end{aligned}
\end{equation}

Lastly, we give the estimate for $\div \mathcal{K}(a)$. 
Since it holds that 
$$
 \int_0^t
 \langle t-\t\rangle^{-(\frac{d}{2p'}+\frac{s}{2})}\||\N|^{-1}\N\Delta a^2\|_{\fB_{p,\infty}^{1-\frac{d}{p'}}}^\ell 
 d\t
 \les \int_0^t
 \langle t-\t\rangle^{-(\frac{d}{2p'}+\frac{s}{2})}\|a^2\|_{\fB_{p,\infty}^{1-\frac{d}{p'}}}^\ell 
 d\t, 
$$
we obtain by similar arguments to \eqref{decay;N_2} that 
\begin{align}
 \int_0^t
 \langle t-\t\rangle^{-(\frac{d}{2p'}+\frac{s}{2})}\||\N|^{-1}\N\Delta a^2\|_{\fB_{p,\infty}^{1-\frac{d}{p'}}}^\ell 
 d\t
 \les \langle t\rangle^{-(\frac{d}{2p'}+\frac{s}{2})}(\mX_p(t)^2+\mD_p(t)^2). 
\end{align}

Applying \eqref{est;p-1} and Lemma \ref{lem;ce} to the estimation of $\N |\N a|^2$, we have 
\begin{align*}
\int_0^t
 \langle t-\t\rangle^{-(\frac{d}{2p'}+\frac{s}{2})}\|\N a \cdot \N a^\ell\|_{\fB_{p,\infty}^{1-\frac{d}{p'}}}^\ell 
 d\t
 \les& \int_0^t
 \langle t-\t\rangle^{-(\frac{d}{2p'}+\frac{s}{2})}
 \|\N a\|_{\cfB^\frac{d}{p}}\|\N a\|_{\fB_{p,\infty}^{1-\frac{d}{p'}}}^\ell
 d\t \\
 \les& \mD_p(t)^2
 \int_0^t
 \langle t-\t\rangle^{-(\frac{d}{2p'}+\frac{s}{2})}
 \langle\t\rangle^{-\min(\frac{d}{2}+\frac{1}{2},\al)}\langle\t\rangle^{-1}
 d\t \\
 \les& \langle t\rangle^{-(\frac{d}{2p'}+\frac{s}{2})} \mD_p(t)^2. 
\end{align*}
For the high-frequency part of $\N |\N a|^2$, in the case $t \le 2$, we obtain by noting that 
$3-d/p' \le 1+d/p$ and by using \eqref{est;p-2} that 
\begin{align*}
\int_0^t\langle t-\t\rangle^{-(\frac{d}{2p'}+\frac{s}{2})}
  \|\N a \cdot \N a^h\|_{\fB_{p,\infty}^{1-\frac{d}{p'}}}^\ell d\t 
\les& \int_0^t\langle t-\t\rangle^{-(\frac{d}{2p'}+\frac{s}{2})}
       \|a\|_{\cfB^\frac{d}{p}}\|\N a^h\|_{\cfB^{2-\frac{d}{p'}}}d\t \\
\les& \int_0^t 
       \|a\|_{\cfB^\frac{d}{p}}\|a\|_{\cfB^{2+\frac{d}{p}}}^h d\t \\
\les& \langle t\rangle^{-(\frac{d}{2p'}+\frac{s}{2})}\mX_p(t)^2. 
\end{align*} 
On the other hand, we also see that if $t>2$,  
\begin{align*}
\int_0^t\langle t-\t\rangle^{-(\frac{d}{2p'}+\frac{s}{2})}
  \|\N a \cdot \N a^h\|_{\fB_{p,\infty}^{1-\frac{d}{p'}}}^\ell d\t 
\les& \left(\int_0^1+\int_1^t\right)\langle t-\t\rangle^{-(\frac{d}{2p'}+\frac{s}{2})}
       \|\N a\|_{\cfB^{-1+\frac{d}{p}}}\|a\|_{\cfB^{3-\frac{d}{p'}}}^hd\t \\
\les& \langle t\rangle^{-(\frac{d}{2p'}+\frac{s}{2})}\mX_p(1)^2\\
     &+\mD_p(t)^2 \int_1^t 
        \langle t-\t\rangle^{-(\frac{d}{2p'}+\frac{s}{2})}
        \langle\t\rangle^{-\min(\frac{d}{2},\al)}
        \langle\t\rangle^{-\al}\cdot \frac{\langle\t\rangle^\al}{\t^\al} 
      d\t \\
\les& \langle t\rangle^{-(\frac{d}{2p'}+\frac{s}{2})}(\mX_p(t)^2+\mD_p(t)^2). 
\end{align*}
Gathering all, we thus obtain 
\begin{equation} \label{decay;N_5}
\begin{aligned}
\int_0^t\langle t-\t\rangle^{-(\frac{d}{2p'}+\frac{s}{2})}
 \||\N|^{-1}\N |\N a|^2\|_{\fB_{p,\infty}^{1-\frac{d}{p'}}}^\ell d\t
\les \langle t\rangle^{-(\frac{d}{2p'}+\frac{s}{2})}(\mX_p(t)^2+\mD_p(t)^2). 
\end{aligned}
\end{equation}
In a similar way to obtaining \eqref{decay;N_5}, we immediately obtain that 
\begin{equation} \label{decay;N_6}
\begin{aligned}
\int_0^t\langle t-\t\rangle^{-(\frac{d}{2p'}+\frac{s}{2})}
 \||\N|^{-1}\div (\N a \otimes \N a)\|_{\fB_{p,\infty}^{1-\frac{d}{p'}}}^\ell d\t
\les \langle t\rangle^{-(\frac{d}{2p'}+\frac{s}{2})}(\mX_p(t)^2+\mD_p(t)^2). 
\end{aligned}
\end{equation}

Combining \eqref{decay;N_1}-\eqref{decay;N_4}, \eqref{decay;N_5}
 and \eqref{decay;N_6}, we obtain the desired estimate. 
\end{proof}

\subsection{The time-weighted estimates for the high-frequency part of $(\N a,m)$} \label{subsect;tw_high} 

By using \eqref{est;F-L} and mimicking the arguments as in \cite{Ka-Sh-Xu, Na2} (for details, see the proof of Lemma 4.3 in \cite{Ka-Sh-Xu} or \S 5.2 in \cite{Na2}),
 we are able to obtain that  for all $t>0$, 
\begin{equation} \label{decay;high_am}
\begin{aligned}
 \|\langle\t\rangle^\al(|\N|a,m)\|_{\wt{L^\infty(I};\cfB^{-1+\frac{d}{p}})}^h
 &+\|\t^\al(|\N|a,m)\|_{\wt{L^\infty(I};\cfB^{-1+\frac{d}{p}})}^h \\
 &\les \mX_{p,0} +\|\t^\al N(a,m)\|_{\wt{L^\infty(I};\cfB^{-1+\frac{d}{p}})}^h. 
\end{aligned}
\end{equation}
Here we recall that $\al=d-\ve$ and $I=(0,t)$.

In order to complete the nonlinear estimate, 
we introduce the Kozono-Shimada type bilinear estimate (cf. \cite{Ka-Sh-Xu, Ko-Sh}). 
The proof of the following lemma is given by Lemma 5.6 in \cite{Na2}. 
  
\begin{lem}[\cite{Na2}] \label{lem;KS_type}
Let $d \ge 1$, $1 \le p,r,r_i \le \infty$ $(i=1,2,3,4)$ and $T>0$. 
Assume that $1/r=1/r_1+1/r_2=1/r_3+1/r_4$, then it holds that there exists some $C>0$ such that 
\begin{align} \label{est;KS_type}
\|fg\|_{\wt{L^r_T(}\cfB^\frac{d}{p})}
\le C\left(\|f\|_{\wt{L^{r_1}_T(}\fB_{\infty,\infty}^{-1})}\|g\|_{\wt{L^{r_2}_T(}\cfB^{1+\frac{d}{p}})}
     +\|f\|_{\wt{L^{r_3}_T(}\cfB^{1+\frac{d}{p}})}\|g\|_{\wt{L^{r_4}_T(}\fB_{\infty,\infty}^{-1})}\right). 
\end{align}
\end{lem}

\begin{prop}[{\it The estimate for the high frequencies}] \label{prop;high_N}
Let $1 \le p \le 2d/(d-1)$ and $\al=d-\ve$ with some small $\ve>0$. There exists some constant $C>0$ such that 
for any $t>0$, 
\begin{align*}
\|\t^\al N(a,m)\|_{\wt{L^\infty(I};\cfB^{-1+\frac{d}{p}})}^h
\le C(1+\mX_p(t))(\mX_p(t)^2+\mD_p(t)^2). 
\end{align*}
\end{prop} 

\begin{proof}[The proof of Proposition \ref{prop;high_N}] 
Firstly, we consider the estimation of $\div ((I(a)-1)m \otimes m)$. 
Thanks to \eqref{est;Banach_ring} and \eqref{est;1/1+r}, we obtain that  
\begin{align*}
\|\t^\al \div ((I(a)-1)m \otimes m)\|_{\wt{L^\infty_t(}\cfB^{-1+\frac{d}{p}})}^h
\les& \|\t^\al (I(a)-1)m \otimes m\|_{\wt{L^\infty_t(}\cfB^\frac{d}{p})}^h \\
\les& (1+\mX_p(t)) \times \\
     &\left(\|\t^{\frac{d}{2}-\frac{\ve}{2}}m^\ell\|_{\wt{L^\infty_t(}\cfB^\frac{d}{p})}
      \|\t^{\frac{d}{2}-\frac{\ve}{2}}m^\ell\|_{\wt{L^\infty_t(}\cfB^\frac{d}{p})} \right. \\
    &\left. +\|m^\ell\|_{\wt{L^\infty_t(}\cfB^\frac{d}{p})}
      \|\t^\al m^h\|_{\wt{L^\infty_t(}\cfB^\frac{d}{p})} \right. \\
    &\left. +\|\t^\al m^h\otimes m^h\|_{\wt{L^\infty_t(}\cfB^\frac{d}{p})} \right). 
\end{align*}
Applying \eqref{est;KS_type} in Lemma \ref{lem;KS_type} to the last term and using the embedding 
$\cfB^{-1+d/p}(\R^d) \hr \fB_{\infty,\infty}^{-1}(\R^d)$ obtained by Lemma \ref{lem;sm}, we see that 
\begin{align*}
\|\t^\al m^h \otimes m^h\|_{\wt{L^\infty_t(}\cfB^\frac{d}{p})}
\les \|m^h\|_{\wt{L^\infty_t(}\fB_{\infty,\infty}^{-1})}
      \|\t^\al m^h\|_{\wt{L^\infty_t(}\cfB^{1+\frac{d}{p}})} 
\les \mX_p(t)^2+\mD_p(t)^2. 
\end{align*}
Combining the above estimates and 
\begin{equation}
\begin{aligned} \label{est;komakai}
&\|\t^{\frac{d}{2}-\frac{\ve}{2}}m^\ell\|_{\wt{L^\infty_t(}\cfB^\frac{d}{p})}
\les \|\langle\t\rangle^{\frac{d}{2}-\frac{\ve}{2}} m\|_{L^\infty(I;\cfB^{\frac{d}{p}-\ve})}^\ell
\les \mD_p(t), \\ 
&\|\t^\al m^h\|_{\wt{L^\infty_t(}\cfB^\frac{d}{p})}+\|\t^\al m^h\|_{\wt{L^\infty_t(}\cfB^{1+\frac{d}{p}})} \les \|\t^\al m\|_{\wt{L^\infty_t(}\cfB^{1+\frac{d}{p}})}^h \les \mD_p(t), 
\end{aligned}
\end{equation}
we obtain that 
\begin{align} \label{est;high_N1}
\|\t^\al \div ((I(a)-1)m \otimes m)\|_{\wt{L^\infty_t(}\cfB^{-1+\frac{d}{p}})}^h 
\les \mX_p(t)^2+\mD_p(t)^2. 
\end{align}

On the other hand, it follows from \eqref{est;Banach_ring}, \eqref{est;I_P} and Young's inequality that 
\begin{equation} \label{est;high_N2} 
\begin{aligned}
\|\t^\al \N (\wt{I}_P(a) a^2)\|_{\wt{L^\infty_t(}\cfB^{-1+\frac{d}{p}})}^h
\les&\|\t^{\frac{d}{2}-\frac{\ve}{2}}a^\ell\|_{\wt{L^\infty_t(}\cfB^{\frac{d}{p}})}
      \|\t^{\frac{d}{2}-\frac{\ve}{2}}a^\ell\|_{\wt{L^\infty_t(}\cfB^{\frac{d}{p}})} \\
    &+\|\t^\al a^h\|_{\wt{L^\infty_t(}\cfB^{\frac{d}{p}})}
      \|a^\ell\|_{\wt{L^\infty_t(}\cfB^{\frac{d}{p}})} \\
    &+\|a\|_{\wt{L^\infty_t(}\cfB^{\frac{d}{p}})}
      \|\t^\al a^h\|_{\wt{L^\infty_t(}\cfB^{\frac{d}{p}})} \\
\les& \mX_p(t)^2+\mD_p(t)^2,  
\end{aligned}
\end{equation}
where we applied the same argument as \eqref{est;I_P} to 
$
\wt{I}_P(a)a=\wt{I}_P(a) a^\ell+ \wt{I}_P(a)a^h. 
$

As for the estimation of $\mathcal{L}(I(a) m)$, we obtain by noting that 
$\N(I(a)m)= (\N I(a)) \otimes m+I(a) \N m$ and using \eqref{est;Banach_ring}, \eqref{est;1/1+r} that 
\begin{align*}
\|\t^\al \Delta (I(a)m)\|_{\wt{L^\infty_t(}\cfB^{-1+\frac{d}{p}})}^h
\les& \|\t^\al \N (I(a) m)\|_{\wt{L^\infty_t(}\cfB^{\frac{d}{p}})}^h \\
\les& \|\t^\al (\N I(a)) \otimes m\|_{\wt{L^\infty_t(}\cfB^{\frac{d}{p}})}^h
      +\|\t^\al I(a) \N m\|_{\wt{L^\infty_t(}\cfB^{\frac{d}{p}})}^h. 
\end{align*}
By the similar arguments as in \eqref{est;Na}, we have for $\al=0$ or $1$, 
\begin{equation} \label{est;Na2}
 \left\|\frac{a^\ell}{1+a}\right\|_{\wt{L^\infty_t(}\cfB^{\al+\frac{d}{p}})} 
 \les \||\N|^\al a^\ell\|_{\wt{L^\infty_t(}\cfB^{\frac{d}{p}})}, \quad 
 \left\|\frac{a^h}{1+a}\right\|_{\wt{L^\infty_t(}\cfB^{\al+\frac{d}{p}})}
  \les \||\N|^\al a^h\|_{\wt{L^\infty_t(}\cfB^{\frac{d}{p}})}. 
\end{equation}
In a similar way to obtaining \eqref{est;high_N1} and \eqref{est;high_N2}, it follows from \eqref{est;Banach_ring} and \eqref{est;Na2} with $\al=1$ that 
\begin{align*}
\|\t^\al (\N I(a)) \otimes m\|_{\wt{L^\infty_t(}\cfB^{\frac{d}{p}})}
\les& \|\t^{\frac{d}{2}-\frac{\ve}{2}}\N a^\ell\|_{\wt{L^\infty_t(}\cfB^{\frac{d}{p}})}
      \|\t^{\frac{d}{2}-\frac{\ve}{2}} m^\ell \|_{\wt{L^\infty_t(}\cfB^{\frac{d}{p}})} \\
    &+\|\t^\al \N a^h\|_{\wt{L^\infty_t(}\cfB^{\frac{d}{p}})}
      \|m^\ell\|_{\wt{L^\infty_t(}\cfB^{\frac{d}{p}})} \\
    &+\|\N a^\ell\|_{\wt{L^\infty_t(}\cfB^{\frac{d}{p}})}
      \|\t^\al m^h\|_{\wt{L^\infty_t(}\cfB^{\frac{d}{p}})} \\
    &+\|\t^\al \N a^h \otimes m^h\|_{\wt{L^\infty_t(}\cfB^{\frac{d}{p}})}, 
\end{align*}
Applying Lemma \ref{lem;KS_type} and $\cfB^{-1+d/p}(\R^d) \hr \fB_{\infty,\infty}^{-1}(\R^d)$ to the last term, it holds that 
\begin{align*}
\|\t^\al \N a^h \otimes m^h\|_{\wt{L^\infty_t(}\cfB^{\frac{d}{p}})}
\les& \|\N a^h\|_{\wt{L^\infty_t(}\fB_{\infty,\infty}^{-1})}
      \|\t^\al m^h\|_{\wt{L^\infty_t(}\cfB^{1+\frac{d}{p}})} \\
   &+\|\t^\al \N a^h\|_{\wt{L^\infty_t(}\cfB^{1+\frac{d}{p}})}
     \|m^h\|_{\wt{L^\infty_t(}\fB_{\infty,\infty}^{-1})} 
\les \mX_p(t)^2+\mD_p(t)^2. 
\end{align*}
On the other hand, in a similar way to obtaining \eqref{est;high_N2}, we obtain by \eqref{est;Na2} with $\al=0$ that  
\begin{align*}
\|\t^\al I(a) \N m\|_{\wt{L^\infty_t(}\cfB^{\frac{d}{p}})}
\les& \|\t^{\frac{d}{2}-\frac{\ve}{2}}a^\ell\|_{\wt{L^\infty_t(}\cfB^{\frac{d}{p}})}
      \|\t^{\frac{d}{2}-\frac{\ve}{2}} \N m^\ell \|_{\wt{L^\infty_t(}\cfB^{\frac{d}{p}})} \\
    &+\|\t^\al a^h\|_{\wt{L^\infty_t(}\cfB^{\frac{d}{p}})}
      \|\N m^\ell\|_{\wt{L^\infty_t(}\cfB^{\frac{d}{p}})} \\
    &+\|a\|_{\wt{L^\infty_t(}\cfB^{\frac{d}{p}})}
      \|\t^\al m^h\|_{\wt{L^\infty_t(}\cfB^{\frac{d}{p}})} 
\les \mX_p(t)^2+\mD_p(t)^2. 
\end{align*}
Therefore we obtain from combining the above estimates that 
\begin{align} \label{est;high_N3}
\|\t^\al \Delta (I(a)m)\|_{\wt{L^\infty_t(}\cfB^{-1+\frac{d}{p}})}^h
\les \mX_p(t)^2+\mD_p(t)^2.
\end{align}
Analogously, it holds that 
\begin{align} \label{est;high_N4}
\|\t^\al \N \div (I(a)m)\|_{\wt{L^\infty_t(}\cfB^{-1+\frac{d}{p}})}^h
\les \mX_p(t)^2+\mD_p(t)^2.
\end{align}
Lastly, let us consider the estimation of $\div \mathcal{K}(a)$. 
Since $\Delta a^2=2|\N a|^2+2a \Delta a$, it holds that 
\begin{align*}
\|\t^\al \N \Delta a^2\|_{\wt{L^\infty_t(}\cfB^{-1+\frac{d}{p}})}^h
\les \|\t^\al \Delta a^2\|_{\wt{L^\infty_t(}\cfB^{\frac{d}{p}})}^h 
\les \|\t^\al |\N a|^2\|_{\wt{L^\infty_t(}\cfB^{\frac{d}{p}})}
      +\|\t^\al a \Delta a\|_{\wt{L^\infty_t(}\cfB^{\frac{d}{p}})}. 
\end{align*}
Applying \eqref{est;Banach_ring} and Lemma \ref{lem;KS_type} to the first term $|\N a|^2$, we see that 
\begin{align*}
\|\t^\al |\N a|^2\|_{\wt{L^\infty_t(}\cfB^{\frac{d}{p}})}
\les& \|\t^{\frac{d}{2}-\frac{\ve}{2}}\N a^\ell\|_{\wt{L^\infty_t(}\cfB^{\frac{d}{p}})}^2
      +\|\t^\al \N a^h\|_{\wt{L^\infty_t(}\cfB^{\frac{d}{p}})}
       \|\N a^\ell\|_{\wt{L^\infty_t(}\cfB^{\frac{d}{p}})} \\
    &+\|\t^\al \N a^h\|_{\wt{L^\infty_t(}\fB^{-1}_{\infty,\infty})}
       \|\N a^h\|_{\wt{L^\infty_t(}\cfB^{1+\frac{d}{p}})} \\
\les& \mX_p(t)^2+\mD_p(t)^2. 
\end{align*} 
On the other hand, it follows from \eqref{est;Banach_ring} that  
\begin{align*}
\|\t^\al a \Delta a\|_{\wt{L^\infty_t(}\cfB^{\frac{d}{p}})}
\les& \|\t^{\frac{d}{2}-\frac{\ve}{2}} a^\ell\|_{\wt{L^\infty_t(}\cfB^{\frac{d}{p}})}
      \|\t^{\frac{d}{2}-\frac{\ve}{2}} \Delta a^\ell\|_{\wt{L^\infty_t(}\cfB^{\frac{d}{p}})} \\
    &+\|\t^\al a^h\|_{\wt{L^\infty_t(}\cfB^{\frac{d}{p}})}
       \|\Delta a^\ell\|_{\wt{L^\infty_t(}\cfB^{\frac{d}{p}})} 
     +\|\t^\al a\|_{\wt{L^\infty_t(}\cfB^{\frac{d}{p}})}
       \|\t^\al \Delta a^h\|_{\wt{L^\infty_t(}\cfB^{\frac{d}{p}})} \\
\les& \mX_p(t)^2+\mD_p(t)^2. 
\end{align*} 
Gathering the above estimates, we also obtain that 
\begin{equation} \label{est;high_N5}
\|\t^\al \N \Delta a^2\|_{\wt{L^\infty_t(}\cfB^{-1+\frac{d}{p}})}^h
\les \mX_p(t)^2+\mD_p(t)^2. 
\end{equation}
Analogously, it holds that 
\begin{align}
&\|\t^\al \N |\N a|^2\|_{\wt{L^\infty_t(}\cfB^{-1+\frac{d}{p}})}^h
\les \mX_p(t)^2+\mD_p(t)^2. \label{est;high_N6} \\
&\|\t^\al \div (\N a \otimes \N a)\|_{\wt{L^\infty_t(}\cfB^{-1+\frac{d}{p}})}^h
\les \mX_p(t)^2+\mD_p(t)^2. \label{est;high_N7}
\end{align}
Combining \eqref{est;high_N1}-\eqref{est;high_N7}, we thus obtain the desired estimate. 
\end{proof}

\subsection{The proof of \eqref{est;decay} in Theorem \ref{thm;Lp-L1}} \label{subsect;comp}
Applying Propositions \ref{prop;decay_low_N} to \eqref{decay;low_ruB} and Proposition \ref{prop;high_N} to \eqref{decay;high_am} respectively, 
we obtain that 
\begin{align*}
\mD_p(t) \les \mX_{p,0}+\mD_{p,0}+(1+\mX_p(t))(\mX_p(t)^2+\mD_p(t)^2). 
\end{align*}
Therefore, the above estimate together with $\mX_{p,0}+\mD_{p,0} \ll 1$ and 
Propostion \ref{prop;unif} leads to our assertion, 
namely, $\mD_p(t) \les 1$. 
Lastly, we would like to discuss the derivation of \eqref{est;decay} 
in Theorem \ref{thm;Lp-L1} by using the boundedness of $\mD_p(t)$. 
For all $t \ge 1$ and $-d/p'<s \le 1+d/p$, 
\begin{equation} \label{est;asert_decay}
\begin{aligned}
t^{\frac{d}{2p'}+\frac{s}{2}}\||\N|^s a(t)\|_{\cfB^0}
\les& t^{\frac{d}{2p'}+\frac{s}{2}}\left(\||\N|^s a^\ell(t)\|_{\cfB^0}+\||\N|^s a^h(t)\|_{\cfB^0}\right) \\
\les& \langle t\rangle^{\frac{d}{2p'}+\frac{s}{2}}\|a(t)\|_{\cfB^s}^\ell
      +t^\al \|a(t)\|_{\cfB^{2+\frac{d}{p}}}^h\\
\les& \|\langle\t\rangle^{\frac{d}{2p'}+\frac{s}{2}}a\|_{L^\infty(1,t;\cfB^s)}^\ell
      +\|\t^\al a\|_{L^\infty(1,t;\cfB^{2+\frac{3}{p}})}^h 
\les \mD_p(t),   
\end{aligned}
\end{equation}
where we used the relation $\t^{\frac{d}{2p'}+\frac{s}{2}} \les \t^\al$ for all $\t \ge 1$. 
Therefore, it follows from the above estimate that $\||\N|^sa(t)\|_{\cfB^0}=O(t^{-(\frac{d}{2p'}+\frac{s}{2})})$ $(t \to \infty)$. 
Similarly, we also obtain $\||\N|^sm(t)\|_{\cfB^0}=O(t^{-(\frac{d}{2p'}+\frac{s}{2})})$ $(t \to \infty)$. 
Therefore, we complete the proof of Theorem \ref{thm;Lp-L1}. 

\subsection{The proof of the linear approximation \eqref{est;lin_approx} in Theorem \ref{thm;Lp-L1}}
To obtain \eqref{est;lin_approx}, we need to handle the low-frequencies and high-frequencies of 
\begin{align} \label{eqn;lin_approx}
\left(
\begin{array}{@{\,}c@{\,}}
a\\
m
\end{array}
\right)(t)-G(t,\cdot)*
\left(
\begin{array}{@{\,}c@{\,}}
a_0\\
m_0
\end{array}
\right)
=\int_0^t
G(t-\t,\cdot)*
\left(
\begin{array}{@{\,}c@{\,}}
0\\
N(a,m)
\end{array}
\right)
d\t. 
\end{align}
Hereafter, we denote the linearized solution $G(t,\cdot)*{}^t(a_0,m_0)$ by ${}^t(a^{lin},m^{lin})$. Notice that 
${}^t(a^{lin},m^{lin})$ satisfies the problem \eqref{eqn;mNSK2} with $N(a,m) \equiv 0$, 
the differences $(\wt{a},\wt{m}):=(a-a^{lin},m-m^{lin})$ which is identified as the left hand side of 
\eqref{eqn;lin_approx} fulfills the followings:
\begin{equation} \label{eqn;til}
\left\{
\begin{aligned}
&\pt_t \wt{a}+\div \wt{m}=0, \\
&\pt_t \wt{m}-\mathcal{L} \wt{m}+\gm^2 \N \wt{a} -\kappa \N \Delta \wt{a}=N(a,m), \\
&(\wt{a},\wt{m})|_{t=0}=(0,0). 
\end{aligned}
\right. 
\end{equation}
Performing the same calculation as we show \eqref{decay;high_am} and applying Proposition \ref{prop;high_N}, 
one can obtain the following estimate for the 
high frequencies of ${}^t(\wt{a},\wt{m})$:
\begin{equation} \label{decay;high_til}
\begin{aligned}
 \|\langle\t\rangle^\al(|\N|\wt{a},\wt{m})\|_{\wt{L^\infty_t}(\cfB^{-1+\frac{d}{p}})}^h
 +\|\t^\al(|\N|\wt{a},\wt{m})\|_{\wt{L^\infty_t}(\cfB^{1+\frac{d}{p}})}^h 
 \les& \|\t^\al N(a,m)\|_{\wt{L^\infty_t}(\cfB^{-1+\frac{d}{p}})}^h \\
 \les& \mX_p(t)^2+\mD_p(t)^2. 
\end{aligned}
\end{equation}
Since $\mD_p(t) \les 1$ is obtained by the proof of \eqref{est;decay} in \S \ref{subsect;comp} and 
$d/2p'+(s+1)/2<\al$, it follows from 
\eqref{decay;high_til} that for all $t>1$ and $-d/p'<s\le d/p$,  
\begin{equation} \label{est;diff_high}
\begin{aligned}
\|(\wt{a},\wt{m})(t)\|_{\cfB^s}^h
\les& t^{-(\frac{d}{2p'}+\frac{s+1}{2})}t^{(\frac{d}{2p'}+\frac{s+1}{2})}
      \|(|\N|\wt{a},\wt{m})(t)\|_{\cfB^{1+\frac{d}{p}}}^h \\
\les& t^{-(\frac{d}{2p'}+\frac{s+1}{2})}
      \|\t^\al(|\N|\wt{a},\wt{m})\|_{\wt{L^\infty_t}(\cfB^{1+\frac{d}{p}})}^h 
\les t^{-(\frac{d}{2p'}+\frac{s+1}{2})}. 
\end{aligned}
\end{equation}
Therefore, it suffices to lead the estimate for the low-frequency part of ${}^t(\wt{a},\wt{m})$. 
For the identity \eqref{eqn;lin_approx}, 
it follows from the triangle inequality that 
\begin{equation} \label{est;diff}
\begin{aligned}
\left\|
\left(
\begin{array}{@{\,}c@{\,}}
\wt{a}\\
\wt{m}
\end{array}
\right)\hspace{-0.6mm}(t)
\right\|_{\cfB^s}^\ell 
\les& \left\|
      \int_0^{t/2}
      G(t-\t,\cdot)*
      \left(
      \begin{array}{@{\,}c@{\,}}
       0\\
       N(a,m)
      \end{array}
      \right)
      d\t
      \right\|_{\cfB^s}^\ell \\
   &+\left\|
      \int_{t/2}^{t}
      G(t-\t,\cdot)*
      \left(
      \begin{array}{@{\,}c@{\,}}
       0\\
       N(a,m)
      \end{array}
      \right)
      d\t
      \right\|_{\cfB^s}^\ell
=:\wt{I}_1+\wt{I}_2. 
\end{aligned}
\end{equation}
Thanks to Proposition \ref{lem;pw_L} with $f,g \equiv 0$, it holds that 
\begin{align} \label{est;pw_N}
\left|
\mathcal{G}(t,\xi)
\left(
\begin{array}{@{\,}c@{\,}}
0\\
\wh{\phi}_j \wh{N(a,m)}
\end{array}
\right)
\right|
\les e^{-\dl_0t|\xi|^2}|\wh{\phi}_j\wh{N(a,m)}|
\end{align}
which yields to 
\begin{align*}
\langle t\rangle^{\frac{d}{2p'}+\frac{s+1}{2}}
\left\|
G(t,\cdot)*
\left(
\begin{array}{@{\,}c@{\,}}
0\\
N(a,m)
\end{array}
\right)
\right\|_{\cfB^s}^\ell 
\les& \langle t\rangle^{\frac{d}{2p'}+\frac{s+1}{2}} 
      \sum_{j \le j_0} 2^{sj} e^{-\dl_0t2^{2j}}\|\wh{\phi}_j\wh{N(a,m)}\|_{L^{p'}} \\
\les& \||\N|^{-1}N(a,m)\|_{\fB_{p,\infty}^{-\frac{d}{p'}}}^\ell
      \left(1+\sum_{j\le j_0}(t^\frac{1}{2}2^{j})^{\frac{d}{p'}+s+1} e^{-\dl_0t2^{2j}}\right). 
\end{align*}
It follows from Lemma \ref{lem;ce} that 
\begin{align} \label{est;G*U_0}
\left\|
G(t,\cdot)*
\left(
\begin{array}{@{\,}c@{\,}}
0\\
N(a,m)
\end{array}
\right)
\right\|_{\cfB^s}^\ell 
\les \langle t\rangle^{-(\frac{d}{2p'}+\frac{s+1}{2})}\||\N|^{-1}N(a,m)\|_{\fB_{p,\infty}^{-\frac{d}{p'}}}^\ell. 
\end{align}
As for $\wt{I}_1$ in \eqref{est;diff}, we obtain by using \eqref{est;G*U_0} that
\begin{align} \label{est;tl_I1}
\wt{I}_1
\les \int_0^{t/2}
       \langle t-\t\rangle^{-(\frac{d}{2p'}+\frac{s+1}{2})}
       \||\N|^{-1}N(a,m)\|_{\fB_{p,\infty}^{-\frac{d}{p'}}}^\ell
      d\t. 
\end{align}  
On the other hand, it follows from \eqref{est;pw_N} that
\begin{align} \label{est;tl_I2}
\wt{I}_2
\les \int_{t/2}^t \sum_{j\le j_0}2^{sj}e^{-\dl_0(t-\t)2^{2j}}\|\wh{\phi}_j\wh{N(a,m)}\|_{L^{p'}} d\t
\les \int_{t/2}^t \|N(a,m)\|_{\cfB^s}^\ell d\t. 
\end{align}
Combining \eqref{est;diff}, \eqref{est;tl_I1} and \eqref{est;tl_I2}, we see that 
\begin{equation} \label{est;diff2}
\begin{aligned}
\left\|
\left(
\begin{array}{@{\,}c@{\,}}
\wt{a}\\
\wt{m}
\end{array}
\right) \hspace{-0.5mm}(t)
\right\|_{\cfB^s}^\ell 
\les& \int_0^{t/2}
       \langle t-\t\rangle^{-(\frac{d}{2p'}+\frac{s+1}{2})}
       \||\N|^{-1}N(a,m)\|_{\fB_{p,\infty}^{-\frac{d}{p'}}}^\ell
      d\t \\
   &+\int_{t/2}^t \|N(a,m)\|_{\cfB^s}^\ell d\t. 
\end{aligned}
\end{equation}

In order to show \eqref{est;lin_approx}, we need to prove the following Propositions: 
\begin{prop} \label{prop;0_t/2} 
Let $d \ge 2$, $1 \le p \le 2$ and $s \in \R$ with $-d/p'<s\le d/p$. 
There exists some constant $C>0$ such that for all $t>2$, 
\begin{equation*}
\int_0^{t/2}
       \langle t-\t\rangle^{-(\frac{d}{2p'}+\frac{s+1}{2})}
       \||\N|^{-1}N(a,m)\|_{\fB_{p,\infty}^{-\frac{d}{p'}}}^\ell
      d\t
\le C t^{-(\frac{d}{2p'}+\frac{s+1}{2})} \delta(t),  
\end{equation*}
where $\delta(t)$ is given by $\log t$ if $d=2$ and $1$ if $d \ge 3$.  
\end{prop}

\begin{prop} \label{prop;t/2_t}
Let $d \ge 2$, $1 \le p \le 2$ and $s \in \R$ with $-d/p'<s\le d/p$. 
There exists some constant $C>0$ such that for all $t>2$, 
\begin{equation*}
\int_{t/2}^t \|N(a,m)\|_{\cfB^s}^\ell d\t
\le C t^{-(\frac{d}{2p'}+\frac{s+1}{2})}.  
\end{equation*}
\end{prop}
If we show the above Propositions \ref{prop;0_t/2}, \ref{prop;t/2_t},  
then we see that it follows from \eqref{est;diff_high} and \eqref{est;diff2} that for any $t > 1$ and 
$-d/p'<s \le d/p$, 
\begin{align}
\|(\wt{a},\wt{m})(t)\|_{\cfB^s}
\les \|(\wt{a},\wt{m})(t)\|_{\cfB^s}^\ell + \|(\wt{a},\wt{m})(t)\|_{\cfB^s}^h
\les t^{-(\frac{d}{2p'}+\frac{s+1}{2})} \delta(t). 
\end{align}
This is the desired estimate \eqref{est;lin_approx}. In order to complete the proof of \eqref{est;lin_approx}, 
we need to consider the proof of Propositions \ref{prop;0_t/2}, \ref{prop;t/2_t}. 

\begin{proof}[The proof of Proposition \ref{prop;0_t/2}]
Throughout the proof, we often use the following product laws which are derived from Lemma \ref{lem;cA-P}: 
\begin{align}
&\|fg\|_{\fB_{p,\infty}^{-\frac{d}{p'}}}
\les \|f\|_{\cfB^\frac{d}{p}}\|g\|_{\fB^{-\frac{d}{p'}}_{p,\infty}}, \label{est;p-3}\\
&\|fg\|_{\fB_{p,\infty}^{-\frac{d}{p'}}}
\les \|f\|_{\cfB^{-1+\frac{d}{p}}}\|g\|_{\fB_{p,\infty}^{1-\frac{d}{p'}}} \label{est;p-4}
\end{align}
for $1 \le p \le 2$. Thanks to  
\eqref{est;1/1+r}, \eqref{est;p-3}, \eqref{est;p-4} and $\mD_p(t) \les 1$, we see that 
\begin{equation} \label{0_t/2;convect}
\begin{aligned}
\int_0^{t/2}
       \langle t-\t\rangle^{-(\frac{d}{2p'}+\frac{s+1}{2})}
       & \||\N|^{-1}\div((I(a)-1)m \otimes m)\|_{\fB_{p,\infty}^{-\frac{d}{p'}}}^\ell
      d\t \\
\les& \int_0^{t/2}
       \langle t-\t\rangle^{-(\frac{d}{2p'}+\frac{s+1}{2})}
       \Big(\|I(a)\|_{\cfB^\frac{d}{p}}+1\Big)\|m \otimes m\|_{\fB_{p,\infty}^{-\frac{d}{p'}}}
      d\t \\
\les& (1+\mX_p(t))\int_0^{t/2}
       \langle t-\t\rangle^{-(\frac{d}{2p'}+\frac{s+1}{2})}
       \|m\|_{\cfB^{-1+\frac{d}{p}}} \|m\|_{\fB_{p,\infty}^{1-\frac{d}{p'}}}
      d\t \\
\les& t^{-(\frac{d}{2p'}+\frac{s+1}{2})} (1+\mX_p(t))\mD_p(t)^2 
      \int_0^{t/2} \langle\t\rangle^{-\frac{d}{2}} d\t
 \les t^{-(\frac{d}{2p'}+\frac{s+1}{2})} \delta(t), 
\end{aligned}
\end{equation}
where we used 
\begin{align*}
&\|m(\t)\|_{\cfB^{-1+\frac{d}{p}}}
\les \|m(\t)\|_{\cfB^{-1+\frac{d}{p}}}^\ell+\|m(\t)\|^h_{\cfB^{-1+\frac{d}{p}}} 
\les \langle\t\rangle^{-(\frac{d}{2}-\frac{1}{2})} \mD_p(t), \\
&\|m(\t)\|_{\fB_{p,\infty}^{1-\frac{d}{p'}}}
\les \|m(\t)\|_{\cfB^{1-\frac{d}{p'}}}^\ell+\|m(\t)\|_{\cfB^{-1+\frac{d}{p}}}^h
\les \langle\t\rangle^{-\frac{1}{2}} \mD_p(t). 
\end{align*}

Since $I_P(a)\N a=\N(\wt{I}_P(a)a^2)$ 
(the definition of $\wt{I}_P$ can be seen in the proof of Lemma \ref{lem;nonlin_est}), 
it follows from \eqref{est;I_P}, \eqref{est;p-3} and \eqref{est;p-4} that 
\begin{equation} \label{0_t/2;pres}
\begin{aligned}
\int_0^{t/2} 
 \langle t-\t\rangle^{-(\frac{d}{2p'}+\frac{s+1}{2})}
 &\||\N|^{-1}\N (\wt{I}_P(a)a^2)\|_{\fB_{p,\infty}^{-\frac{d}{p'}}}^\ell
d\t \\
\les& \int_0^{t/2} 
 \langle t-\t\rangle^{-(\frac{d}{2p'}+\frac{s+1}{2})}
 \|\wt{I}_P(a) a\|_{\cfB^{-1+\frac{d}{p}}}\|a\|_{\fB_{p,\infty}^{1-\frac{d}{p'}}}
d\t \\
\les& \langle t\rangle^{-(\frac{d}{2p'}+\frac{s+1}{2})}
\int_0^{t/2} 
 \|a\|_{\cfB^{-1+\frac{d}{p}}}\|a\|_{\fB_{p,\infty}^{1-\frac{d}{p'}}}
d\t \\
\les& t^{-(\frac{d}{2p'}+\frac{s+1}{2})} \mD_p(t)^2 
      \int_0^{t/2} \langle\t\rangle^{-\frac{d}{2}} d\t
 \les t^{-(\frac{d}{2p'}+\frac{s+1}{2})} \delta(t). 
\end{aligned}
\end{equation}

As for $\mathcal{L}(I(a)m)=\mu \Delta (I(a)m)+(\mu+\lam) \N \div (I(a)m)$, we obtain by 
\eqref{est;1/1+r} and \eqref{est;p-1} that 
\begin{equation} \label{0_t/2;stress_1}
\begin{aligned}
\int_0^{t/2} 
 \langle t-\t\rangle^{-(\frac{d}{2p'}+\frac{s+1}{2})}
 &\||\N|^{-1} \Delta (I(a)m)\|_{\fB_{p,\infty}^{-\frac{d}{p'}}}^\ell
d\t \\
\les& \int_0^{t/2} 
 \langle t-\t\rangle^{-(\frac{d}{2p'}+\frac{s+1}{2})}
 \|I(a)m\|_{\fB_{p,\infty}^{1-\frac{d}{p'}}}^\ell
d\t \\
\les& \int_0^{t/2} 
 \langle t-\t\rangle^{-(\frac{d}{2p'}+\frac{s+1}{2})}
 \|a\|_{\cfB^\frac{d}{p}} \|m\|_{\fB_{p,\infty}^{1-\frac{d}{p'}}}
d\t \\
\les& t^{-(\frac{d}{2p'}+\frac{s+1}{2})} \mD_p(t)^2 
\int_0^{t/2} 
\langle\t\rangle^{-(\frac{d}{2}+\frac{1}{2})}  
d\t 
\les t^{-(\frac{d}{2p'}+\frac{s+1}{2})}. 
\end{aligned}
\end{equation}
Analogously, we see that 
\begin{equation} \label{0_t/2;stress_2}
\int_0^{t/2} 
 \langle t-\t\rangle^{-(\frac{d}{2p'}+\frac{s+1}{2})}
 \||\N|^{-1} \N \div (I(a)m)\|_{\fB_{p,\infty}^{-\frac{d}{p'}}}^\ell
d\t 
\les t^{-(\frac{d}{2p'}+\frac{s+1}{2})}. 
\end{equation}

As for the Korteweg stress tensor $\mathcal{K}(a)$, thanks to \eqref{est;p-4} and $\mD_p(t) \les 1$, we obtain that 
\begin{equation} \label{0_t/2;Korte_1}
\begin{aligned}
\int_0^{t/2} 
 \langle t-\t\rangle^{-(\frac{d}{2p'}+\frac{s+1}{2})}
 \||\N|^{-1} \N \Delta a^2\|_{\fB_{p,\infty}^{-\frac{d}{p'}}}^\ell
d\t 
\les& \int_0^{t/2} 
 \langle t-\t\rangle^{-(\frac{d}{2p'}+\frac{s+1}{2})}
 \|a^2\|_{\fB_{p,\infty}^{-\frac{d}{p'}}}^\ell
d\t \\
\les&  \langle t\rangle^{-(\frac{d}{2p'}+\frac{s+1}{2})}
\int_0^{t/2} 
 \|a\|_{\cfB^{-1+\frac{d}{p}}} \|a\|_{\fB_{p,\infty}^{1-\frac{d}{p'}}}
d\t \\
\les& t^{-(\frac{d}{2p'}+\frac{s+1}{2})} \mD_p(t)^2 
\int_0^{t/2} 
\langle\t\rangle^{-\frac{d}{2}}  
d\t \\
\les& t^{-(\frac{d}{2p'}+\frac{s+1}{2})} \delta(t). 
\end{aligned}
\end{equation}
Analogously, it holds that 
\begin{equation} \label{0_t/2;Korte_2}
\begin{aligned}
&\int_0^{t/2} 
 \langle t-\t\rangle^{-(\frac{d}{2p'}+\frac{s+1}{2})}
 \||\N|^{-1} \N |\N a|^2\|_{\fB_{p,\infty}^{-\frac{d}{p'}}}^\ell
d\t 
\les t^{-(\frac{d}{2p'}+\frac{s+1}{2})} \delta(t), \\
&\int_0^{t/2} 
 \langle t-\t\rangle^{-(\frac{d}{2p'}+\frac{s+1}{2})}
 \||\N|^{-1} \div (\N a \otimes \N a)\|_{\fB_{p,\infty}^{-\frac{d}{p'}}}^\ell
d\t 
\les t^{-(\frac{d}{2p'}+\frac{s+1}{2})} \delta(t). 
\end{aligned}
\end{equation}
Combining \eqref{0_t/2;convect}-\eqref{0_t/2;Korte_2}, we obtain the desired estimate. 
\end{proof} 

\begin{proof}[The proof of Proposition \ref{prop;t/2_t}]
Notice that 
\begin{align*}
\div ((I(a)-1)m \otimes m)
=& (\N I(a) \cdot m)m +(I(a)-1) \,\div (m \otimes m) \\
=& (\N I(a) \cdot m)m+(I(a)-1)\,((m \cdot \N)m+m \,\div m), 
\end{align*}
it follows from Lemma \ref{lem;A-P} with \eqref{est;Banach_ring}, \eqref{est;Na} 
and \eqref{est;asert_decay} that 
\begin{align*}
\int_{t/2}^t \|(\N I(a) \cdot m)m\|_{\cfB^s} d\t
\les& \int_{t/2}^t 
      \|\N I(a)\|_{\cfB^\frac{d}{p}}
      \|m\|_{\cfB^\frac{d}{p}}\|m\|_{\cfB^s} d\t \\
\les& \mD_p(t)^3
      \int_{t/2}^t
       \t^{-(\frac{d}{2p'}+\frac{s}{2})} \t^{-(d+\frac{1}{2})}
       d\t \\
\les& t^{-(\frac{d}{2p'}+\frac{s}{2})}t^{-d-\frac{1}{2}}\frac{t}{2}
\les t^{-(\frac{d}{2p'}+\frac{s+1}{2})}, 
\end{align*}
where we used $d \ge 2$ in the last inequality. Similarly, 
by using \eqref{est;Banach_ring} and \eqref{est;1/1+r}, 
\begin{align*}
\int_{t/2}^t \|(I(a)-1)(m\cdot\N)m\|_{\cfB^s}d\t
\les& (1+\mX_p(t)) \int_{t/2}^t
      \|\N m\|_{\cfB^\frac{d}{p}}\|m\|_{\cfB^s} d\t \\
\les& (1+\mX_p(t))\mD_p(t)^2
      \int_{t/2}^t
       \t^{-(\frac{d}{2p'}+\frac{s}{2})} \t^{-(\frac{d}{2}+\frac{1}{2})}
       d\t \\
\les& t^{-(\frac{d}{2p'}+\frac{s}{2})}t^{-\frac{d}{2}-\frac{1}{2}}\frac{t}{2}
\les t^{-(\frac{d}{2p'}+\frac{s+1}{2})} 
\end{align*}
and 
\begin{align*}
\int_{t/2}^t \|(I(a)-1)\,m \,\div m\|_{\cfB^s}d\t
\les t^{-(\frac{d}{2p'}+\frac{s+1}{2})}. 
\end{align*}
Combining the above estimations, we obtain 
\begin{align} \label{t/2_t;convect_x}
\int_{t/2}^t \|\div ((I(a)-1)m \otimes m)\|_{\cfB^s}d\t
\les t^{-(\frac{d}{2p'}+\frac{s+1}{2})}. 
\end{align}

Lemma \ref{lem;A-P} with \eqref{est;Banach_ring}, \eqref{est;I_P} 
and \eqref{est;asert_decay} give us that 
\begin{equation} \label{t/2_t;pres}
\begin{aligned} 
\int_{t/2}^t \|I_P(a)\N a\|_{\cfB^s}d\t 
\les& \int_{t/2}^t \|I_P(a)\|_{\cfB^\frac{d}{p}} \|\N a\|_{\cfB^s}d\t \\
\les& \mD_p(t)^2 \int_{t/2}^t \t^{-\frac{d}{2}} \t^{-(\frac{d}{2p'}+\frac{s+1}{2})} d\t 
\les t^{-(\frac{d}{2p'}+\frac{s+1}{2})}. 
\end{aligned}
\end{equation}

By virtue of $\N(I(a)m)=\N I(a) \otimes m+I(a)\N m$, it follows from Lemma \ref{lem;A-P} with \eqref{est;Banach_ring}, 
\eqref{est;Na2} and $\mD_p(t) \les 1$ that 
\begin{equation} \label{t/2_t;stress_1}
\begin{aligned} 
\int_{t/2}^t \|\Delta (I(a)m)\|_{\cfB^s}^\ell d\t 
\les& \int_{t/2}^t \|\N (I(a)m)\|_{\cfB^s}^\ell d\t \\
\les& \int_{t/2}^t \|\N a\|_{\cfB^\frac{d}{p}}\|m\|_{\cfB^s} d\t
      +\int_{t/2}^t \|a\|_{\cfB^\frac{d}{p}}\|\N m\|_{\cfB^s} d\t \\
\les& \mD_p(t)^2 \int_{t/2}^t \t^{-\frac{d}{2}} \t^{-(\frac{d}{2p'}+\frac{s+1}{2})} d\t 
\les t^{-(\frac{d}{2p'}+\frac{s+1}{2})}. 
\end{aligned}
\end{equation} 
Analogously, we obtain that 
\begin{equation} \label{t/2_t;stress_2}
\begin{aligned} 
\int_{t/2}^t \|\N \div (I(a)m)\|_{\cfB^s}^\ell d\t 
\les t^{-(\frac{d}{2p'}+\frac{s+1}{2})}. 
\end{aligned}
\end{equation}

As for the Korteweg stress tensor $\mathcal{K}(a)$, thanks to Lemma \ref{lem;A-P} with \eqref{est;Banach_ring} and $\mD_p(t) \les 1$, 
\begin{equation} \label{t/2_t;stress_2}
\begin{aligned} 
\int_{t/2}^t \|\N \Delta a^2\|_{\cfB^s}^\ell d\t 
\les& \int_{t/2}^t \|a \N a\|_{\cfB^s}^\ell d\t \\
\les& \int_{t/2}^t \|\N a\|_{\cfB^\frac{d}{p}} \|a\|_{\cfB^s} d\t \\
\les& \mD_p(t)^2 \int_{t/2}^t \t^{-(\frac{d}{2}+\frac{1}{2})} \t^{-(\frac{d}{2p'}+\frac{s}{2})} d\t
\les t^{-(\frac{d}{2p'}+\frac{s+1}{2})}. 
\end{aligned}
\end{equation} 
Here we used $\N \Del a^2=2 \N \div (a \N a)$ in the first estimation. 
Analogously, it holds that 
\begin{equation} \label{t/2_t;Korte_2}
\begin{aligned}
&\int_0^{t/2} 
 \||\N|^{-1} \N |\N a|^2\|_{\cfB^s}^\ell
d\t 
\les t^{-(\frac{d}{2p'}+\frac{s+1}{2})}, \\
&\int_0^{t/2} 
 \||\N|^{-1} \div (\N a \otimes \N a)\|_{\cfB^s}^\ell
d\t 
\les t^{-(\frac{d}{2p'}+\frac{s+1}{2})}. 
\end{aligned}
\end{equation}
Combining \eqref{t/2_t;convect_x}-\eqref{t/2_t;Korte_2}, we obtain the desired estimate. 
\end{proof}

Applying Propositions \ref{prop;0_t/2} and \ref{prop;t/2_t} to \eqref{est;diff2}, we obtain that 
\begin{align} \label{est;diff_low}
\left\|
\left(
\begin{array}{@{\,}c@{\,}}
a\\
m
\end{array}
\right)(t)-G(t,\cdot)*
\left(
\begin{array}{@{\,}c@{\,}}
a_0\\
m_0
\end{array}
\right)
\right\|_{\cfB^s}^\ell 
\les t^{-(\frac{d}{2p'}+\frac{s+1}{2})} \delta(t). 
\end{align}
Therefore, \eqref{est;diff_high} and \eqref{est;diff_low} yields to the conclusion of 
\eqref{est;lin_approx} in Theorem \ref{thm;Lp-L1}. 

\sect{Diffusion wave properties for global solution} \label{sect;diff_wave}

In this section, we give the proof of Theorem \ref{thm;asympt}. 
Let the initial data $(a_0,m_0)$ satisfy the assumptions \eqref{assump;initial_FB} with $p=2$ and 
\eqref{cond;L1}. In order to show the time-decay of global solution accompanied by 
{\it the diffusion wave property}, let us introduce the modified decay function $\mD(t)$ instead of 
$\mD_2(t)$ as follows:  
\begin{equation*}
\mD(t):=\mD_L^{(1)}(t)+\mD_L^{(2)}(t)+\mD_H(t), 
\end{equation*}
where $\mD_L^{(1)}(t)$, $\mD_L^{(2)}(t)$ and $\mD_{H}(t)$ are defined by 
\begin{align*} 
&\mD_L^{(1)}(t):=\sup_{s \in (-\frac{d}{2},1+\frac{d}{2}]}
           \|\langle\t\rangle^{\frac{d}{4}+\frac{s}{2}} (a_L,m_L)\|_{L^\infty_t(\cB^s)}, \quad 
\mD_L^{(2)}(t):=\|\langle\t\rangle^{q(d)} a_L\|_{L^\infty_t(L^\infty)},  \\
&\mD_H(t):=
        \|\langle\t\rangle^{\al} (\N a,m)\|_{\wt{L^\infty_t(}\cB^{-1+\frac{d}{2}})}^h 
        +\|\t^{\al} (\N a,m)\|_{\wt{L^\infty_t(}\cB^{1+\frac{d}{2}})}^h.  
\end{align*}
The decay rate of high-frequencies $\al$ in $\mD_H(t)$ is denoted by $d-\ve$ with sufficiently small $\ve>0$ and 
$f_L:=\mathcal{F}^{-1}[\chi_L(\xi) \wh{f}\,]$ with the low-frequency cut-off $\chi_L$ as given by Lemma \ref{lem;DW}. 

%

\subsection{The uniform boundedness for the modified decay function $\mD_L ^{(i)}(t)$}  \label{subsect;bdd}
In this section, we shall consider the uniform boundedness for $\mD_L^{(i)}(t)$ $(i=1,2)$. If we obtain such a uniform estimate, then we are able to obtain 
the refined time-decay estimate $\|a(t)\|_{L^\infty}=O(t^{-q(d)})$ as $t \to \infty$. 
First of all, we give the estimation of the low-frequency part of $(a,m)$. 
The solution of the integral equation \eqref{eqn;IE} restricted to low frequencies is denoted by 
\begin{equation} \label{eqn;aLmL}
\left(
\begin{array}{@{\,}c@{\,}}
a_L \\
m_L
\end{array}
\right) (t)
=
\left(
\begin{array}{@{\,}c@{\,}}
G^{1,1}_L(t,\cdot)*a_0+G^{1,2}_L(t,\cdot)*m_0+\dsp \int_0^t G^{1,2}_L(t-\t,\cdot)*N(a,m)\,d\t\\
\dsp G^{2,1}_L(t,\cdot)*a_0+G^{2,2}_L(t,\cdot)*m_0+\int_0^t G^{2,2}_L(t-\t,\cdot)*N(a,m)\,d\t
\end{array}
\right), 
\end{equation}
where $G^{i,j}_L(t,x):=\F^{-1}[\chi_L(\xi) \mG^{i,j}(t,\xi)]$ and we recall that $\mG^{i,j}(t,\xi)$ are defined by \eqref{eqn;mg}. 
By using Lemma \ref{lem;DW} and noting that for $\psi \equiv 1$, 
\begin{gather*}
G_L^{1,1}(t,x)=\pt_t K_{\psi,L}(t,\cdot)-\nu \Delta K_{\psi,L}(t,x), \;\;
G_L^{1,2}(t,x)={}^t \N K_{\psi,L}(t, x)
\end{gather*}
($K_{\psi,L}(t,x)$ is already defined in Lemma \ref{lem;DW}), 
we easily see that 
\begin{align}  \label{est;fund_1}
\|G_L^{1,1}(t,\cdot)\|_{L^\infty}+\|G_L^{1,2}(t,\cdot)\|_{L^\infty} \les \langle t\rangle^{-\frac{(3d-1)}{4}}. 
\end{align}
Applying the Hausdorff-Young inequality and \eqref{est;fund_1} to $\text{\eqref{eqn;aLmL}}_1$ and 
noting that there exists a $C^\infty$-function $\wt{\chi_L}=\wt{\chi_L}(\xi)$ 
whose support is compact in $\R^d$ with $\wt{\chi_L} \chi_L=1$ on $\supp \chi_L$, we obtain that 
\begin{equation} \label{est;basic_1}
\begin{aligned}
\|a_L(t)\|_{L^\infty} 
\les& \|G_L^{1,1}(t,\cdot)\|_{L^\infty} \|a_0\|_{L^1}+\|G_L^{1,2}(t,\cdot)\|_{L^\infty}\|m_0\|_{L^1} \\
      &+\int_0^t \|G^{1,2}_L(t-\t,\cdot)\|_{L^\infty} \|\wt{N_L}(a,m)\|_{L^1} d\t  \\
\les& \langle t\rangle^{-\frac{(3d-1)}{4}} \|(a_0,m_0)\|_{L^1}
      +\int_0^t \langle t-\t\rangle^{-\frac{(3d-1)}{4}} \|\wt{N_L}(a,m)\|_{L^1} d\t,  
\end{aligned}
\end{equation}
where $\wt{N_L}(a,m):=\mathcal{F}^{-1}[\wt{\chi_L} \wh{N(a,m)}]$. 
On the other hand, by similar arguments as \eqref{est;heat}, it follows from Lemma \ref{lem;ce} and Proposition \ref{lem;pw_L} that 
for all $t$, $s$ satisfying $t>0$, $-d/2<s \le 1+d/2$, 
\begin{align} \label{est;basic_2}
\|(a_L,m_L)(t)\|_{\cB^s} 
\les \langle t\rangle^{-(\frac{d}{4}+\frac{s}{2})} \|(a_{0,L},m_{0,L})\|_{\dB_{2,\infty}^{-\frac{d}{2}}}
      +\int_0^t \langle t-\t\rangle^{-(\frac{d}{4}+\frac{s}{2})} \|N_L(a,m)\|_{\dB_{2,\infty}^{-\frac{d}{2}}} d\t. 
\end{align} 

Hereafter, let us consider the estimate for the nonlinear terms in \eqref{est;basic_1} and \eqref{est;basic_2}. 
\begin{prop} \label{prop;ref_decay1} Let $d \ge 2$. 
There exists some constant $C>0$ such that for all $t>0$, 
\begin{align}
\int_0^t \langle t-\t\rangle^{-\frac{(3d-1)}{4}} \|\wt{N}_L(a,m)\|_{L^1} d\t 
\le C\langle t\rangle^{-q(d)} (1+\mX_2(t)+\mD(t))(\mX_2(t)^2+\mD(t)^2). 
\end{align}
\end{prop}

\begin{prop} \label{prop;ref_decay2} Let $d \ge 2$ and $s \in \R$ with $-d/2<s\le 1+d/2$. 
There exists some constant $C>0$ such that for all $t>0$, 
\begin{align}
\int_0^t \langle t-\t\rangle^{-(\frac{d}{4}+\frac{s}{2})} \|N_L(a,m)\|_{\dB_{2,\infty}^{-\frac{d}{2}}} d\t 
\le C\langle t\rangle^{-(\frac{d}{4}+\frac{s}{2})}(1+\mX_2(t))(\mX_2(t)^2+\mD(t)^2). 
\end{align}
\end{prop}

\begin{proof}[The proof of Proposition \ref{prop;ref_decay1}] 
Firstly, we shall mention about some notations. In what follows, $\chi_H=\chi_H(\xi)$ is 
denoted by $1-\chi_L$ and $f_H:=\F^{-1}[\chi_H \wh{f}\,]$ for some $f \in \mathcal{S}'$.  
We notice that  there exists some $j_0 \in \Z$ such that $\|a_H\|_{\cB^s} \les \|a\|^h_{\cB^s}$. 

\noindent
{\it The estimate for $\div((I(a)-1) m\otimes m)=(\N I(a) \cdot m)m+(I(a)-1)\div(m \otimes m)$}: 
First of all, we notice that it follows from $\cB^{d/2}(\R^d) \hr L^\infty(\R^d)$ and the similar argument as \eqref{est;Na2} that 
\begin{equation} \label{est;1/1+r2}
\||\N|^\al (\wt{I}(a)a_A)\|_{L^\infty} \les \||\N|^\al (\wt{I}(a) a_A)\|_{\cB^\frac{d}{2}} \les \||\N|^\al a_A\|_{\cB^\frac{d}{2}}
\end{equation}
for $A=L$ or $H$ and $\al = 0,1$. 
By using the Hausdorff-Young inequality, Sobolev's embedding (see e.g. Proposition 2.20 in \cite{B-C-D}) and \eqref{est;1/1+r2}, we obtain that 
\begin{equation} \label{est;ref1}
\begin{aligned}
\int_0^t \langle t-\t\rangle^{-\frac{(3d-1)}{4}}& \|(\N (\wt{I}(a)a_L) \cdot m)m\|_{L^1} d\t \\
\les& \int_0^t \langle t-\t\rangle^{-\frac{(3d-1)}{4}} \|\N (\wt{I}(a) a_L)\|_{L^\infty}\|m\|_{L^2}^2 d\t \\
\les& \int_0^t \langle t-\t\rangle^{-\frac{(3d-1)}{4}} \langle\t\rangle^{-(\frac{d}{2}+\frac{1}{2})} \|\N a_L\|_{\cB^\frac{d}{2}}\left(\|m_L\|_{\cB^0}+\|m_H\|_{\cB^0}\right)^2 d\t \\
\les& \mD(t)^3 \int_0^t \langle t-\t\rangle^{-\frac{(3d-1)}{4}}  \langle\t\rangle^{-(\frac{d}{2}+\frac{1}{2})} \langle\t\rangle^{-2\min(\frac{d}{4},\al)} d\t \\
\les& \langle t\rangle^{-\frac{(3d-1)}{4}} \mD(t)^3,  
\end{aligned}
\end{equation}
where we used 
$$
  \|m_L\|_{\cB^0} \les \langle\t\rangle^{-\frac{d}{4}} \mD_L^{(1)}(t), \quad \|m_H\|_{\cB^0} \les \langle\t\rangle^{-\al} \mD_H(t).  
$$
When $t \le 2$, it follows from the H\"older inequality and $\langle t\rangle^{q(d)} \les 1$ that 
\begin{equation} \label{est;ref2} 
\begin{aligned}
\int_0^t \langle t-\t\rangle^{-\frac{(3d-1)}{4}}& \|(\N (\wt{I}(a) a_H) \cdot m)m\|_{L^1} d\t \\
\les& \int_0^t \|\N (\wt{I}(a) a_H)\|_{L^\infty} \|m\|_{L^2}^2 d\t \\
\les& \int_0^t \langle\t\rangle^{-\frac{d}{2}} (\langle\t\rangle^{\frac{d}{4}} \|m\|_{\cB^0})^2 \|a\|_{\cB^{2+\frac{d}{2}}}^h d\t 
\les \langle t\rangle^{-q(d)} \mD(t)^2 \mX_2(t). 
\end{aligned}
\end{equation}
When $t > 2$, it follows from H\"older's inequality, \eqref{est;1/1+r2} and $\cB^{d/2}(\R^d) \hr L^\infty(\R^d)$ that 
\begin{equation} \label{est;ref3}
\begin{aligned}
\int_0^t \langle t-\t\rangle^{-\frac{(3d-1)}{4}}& \|(\N (\wt{I}(a)a_H) \cdot m)m\|_{L^1} d\t \\
\les& \left(\int_0^1+\int_1^t \right)  \langle t-\t\rangle^{-\frac{(3d-1)}{4}} \|(\N (\wt{I}(a)a_H) \cdot m)m\|_{L^1} d\t  \\
\les& \langle t\rangle^{-\frac{(3d-1)}{4}} \int_0^1 \left(\frac{\langle\t\rangle}{\langle t-\t\rangle}\right)^{\frac{(3d-1)}{4}} \|\N (\wt{I}(a)a_H)\|_{L^\infty} \|m\|_{L^2}^2 \,d\t \\
&+\int_1^t \langle t-\t\rangle^{-\frac{(3d-1)}{4}} \|\N (\wt{I}(a)a_H)\|_{L^\infty} \|m\|_{L^2}^2 \,d\t \\
\les& \langle t\rangle^{-\frac{(3d-1)}{4}} \mD(t)^2 \int_0^1 \|a\|_{\cB^{2+\frac{d}{2}}}^h d\t \\
&+\mD(t)^3 \int_1^t \langle t-\t\rangle^{-\frac{(3d-1)}{4}}\langle\t\rangle^{-\al} \left(\frac{\langle\t\rangle}{\t} \right)^\al 
 \langle\t\rangle^{-2 \min (\frac{d}{4}, \al)}d\t \\
\les& \langle t\rangle^{-q(d)} (\mD(t)^2 \mX_2(t) +\mD(t)^3).  
\end{aligned}
\end{equation}
Combining \eqref{est;ref1}-\eqref{est;ref3} and applying Proposition \ref{prop;unif}, we obtain 
\begin{align} \label{est;ref_nonlin1}
\int_0^t \langle t-\t\rangle^{-\frac{(3d-1)}{4}}& \|(\N I(a) \cdot m)m\|_{L^1} d\t \les  \langle t\rangle^{-q(d)} (\mD(t)^2 \mX_2(t) +\mD(t)^3). 
\end{align}
On the other hand, since $\div (m \otimes m)=(m \cdot \N)m+m \,\div m$ and \eqref{est;1/1+r2}, we see that 
\begin{equation} \label{est;ref4}
\begin{aligned}
\int_0^t \langle t-\t\rangle^{-\frac{(3d-1)}{4}}& \|(I(a)-1) \div (m \otimes m)\|_{L^1} d\t \\
&\les (1+\mX_2(t)) \int_0^t \langle t-\t\rangle^{-\frac{(3d-1)}{4}} (\|(m \cdot \N)m\|_{L^1}+\|m\, \div m\|_{L^1}) d\t.  
\end{aligned}
\end{equation}
In what follows, we shall indicate only the estimation of $(m \cdot \N)m$. 
By H\"older's inequality and the embedding $\cB^0(\R^d) \hr L^2(\R^d)$, we have 
\begin{equation} \label{est;ref5}
\begin{aligned}
\int_0^t \langle t-\t\rangle^{-\frac{(3d-1)}{4}}& \|(m \cdot \N)m_L\|_{L^1} d\t \\
\les& \int_0^t \langle t-\t\rangle^{-\frac{(3d-1)}{4}} \|m\|_{\cB^0} \|\N m_L\|_{\cB^0}d\t \\
\les& \int_0^t \langle t-\t\rangle^{-\frac{(3d-1)}{4}} (\|m_L\|_{\cB^0}+\|m\|_{\cB^{-1+\frac{d}{2}}}^h) \|\N m_L\|_{\cB^0}d\t \\
\les&  \mD(t)^2 \int_0^t \langle t-\t\rangle^{-\frac{(3d-1)}{4}} \langle\t\rangle^{-\min (\frac{d}{4}, \al)} \langle\t\rangle^{-(\frac{d}{4}+\frac{1}{2})}d\t \\
\les& \mD(t)^2 
         \begin{cases} 
         \langle t\rangle^{-\frac{(3d-1)}{4}}, \quad d=2, \\
         \langle t\rangle^{-(\frac{d}{2}+\frac{1}{2})}, \quad d \ge 3. 
         \end{cases}
\end{aligned}
\end{equation}

When $t \le 2$, it follows from H\"older's inequality and $\cB^{0}(\R^d) \hr L^2(\R^d)$ that 
\begin{equation} \label{est;ref6}
\begin{aligned}
\int_0^t \langle t-\t\rangle^{-\frac{(3d-1)}{4}}& \|(m \cdot \N)m_H\|_{L^1} d\t \\
\les& \int_0^t \|m\|_{\cB^0} \|\N m_H\|_{\cB^0}d\t \\
\les& \mD(t) \int_0^t \|m\|_{\cB^{2+\frac{d}{2}}}^h d\t 
\les \langle t\rangle^{-\frac{(3d-1)}{4}} \mD(t) \mX_2(t). 
\end{aligned}
\end{equation} 
When $t \ge 2$, we obtain by $\langle t\rangle\langle t-\t\rangle^{-1} \les 1$ if $\t \le 1$ that  
\begin{equation} \label{est;ref7}
\begin{aligned}
\int_0^t \langle t-\t\rangle^{-\frac{(3d-1)}{4}}& \|(m \cdot \N)m_H\|_{L^1} d\t \\
\les& \langle t\rangle^{-\frac{(3d-1)}{4}} \int_0^1 \left(\frac{\langle t\rangle}{\langle t-\t\rangle}\right)^{\frac{(3d-1)}{4}} \|m\|_{\cB^0} \|\N m_H\|_{\cB^0}d\t \\
&+ \int_1^t \langle t-\t\rangle^{-\frac{(3d-1)}{4}} \|m\|_{\cB^0} \|\N m_H\|_{\cB^0} d\t \\
\les&  \langle t\rangle^{-\frac{(3d-1)}{4}} \mD(t) \int_0^1 \|m\|_{\cB^{2+\frac{d}{2}}}^h d\t \\
&+ \mD(t)^2 \int_1^t \langle t-\t\rangle^{-\frac{(3d-1)}{4}} \langle\t\rangle^{-\min (\frac{d}{4}, \al)} \langle\t\rangle^{-\al} \left(\frac{\langle\t\rangle}{\t} \right)^\al d\t \\
\les&  \langle t\rangle^{-\frac{(3d-1)}{4}} (\mD(t) \mX_2(t)+\mD(t)^2). 
\end{aligned}
\end{equation} 
Gathering \eqref{est;ref5}, \eqref{est;ref6} and \eqref{est;ref7}, we obtain 
\begin{equation} \label{est;ref8}
\int_0^t \langle t-\t\rangle^{-\frac{(3d-1)}{4}}\|(m \cdot \N)m\|_{L^1} d\t \les \langle t\rangle^{-q(d)}(\mD(t) \mX_2(t)+\mD(t)^2). 
\end{equation}
In a similar argument to obtaining \eqref{est;ref8}, we also obtain that 
\begin{equation} \label{est;ref9}
\int_0^t \langle t-\t\rangle^{-\frac{(3d-1)}{4}}\|m \,\div m\|_{L^1} d\t \les \langle t\rangle^{-q(d)}(\mD(t) \mX_2(t)+\mD(t)^2). 
\end{equation}
Therefore, by combining \eqref{est;ref4}, \eqref{est;ref8} and \eqref{est;ref9}, we arrive at 
\begin{align} \label{est;ref_nonlin2}
\int_0^t \langle t-\t\rangle^{-\frac{(3d-1)}{4}}& \|(I_1(a)-1) \div (m \otimes m)\|_{L^1} d\t 
\les  \langle t\rangle^{-q(d)}(1+\mX_2(t))(\mX_2(t)^2+\mD(t)^2).
\end{align}

\noindent
{\it The estimate for $\mu \Del (I(a)m)=\div(\wt{I}(a)(\N a \otimes m+a \N m))$}: By using the Hausdorff-Young inequality, \eqref{est;1/1+r2} and 
noting that $\|\F^{-1}[|\xi| \chi_L]\|_{L^1} \les 1$, we see that 
\begin{equation} \label{est;ref10}
\begin{aligned}  
\int_0^t \langle t-\t&\rangle^{-\frac{(3d-1)}{4}} \|\F^{-1}[\wt{\chi}_L \F[\Delta (I(a) m)]]\|_{L^1} d\t \\
&\les \int_0^t \langle t-\t\rangle^{-\frac{(3d-1)}{4}} \|\N a \otimes m\|_{L^1} d\t 
         + \int_0^t \langle t-\t\rangle^{-\frac{(3d-1)}{4}} \|a \N m\|_{L^1} d\t. 
\end{aligned}
\end{equation}
Therefore, by the same calculation as in the estimate for $\div((I(a)-1) m\otimes m)$, we are able to obtain that 
\begin{equation} \label{est;ref11}
\begin{aligned}  
\int_0^t \langle t-\t\rangle^{-\frac{(3d-1)}{4}} \|\F^{-1}[\wt{\chi}_L \F[\Delta (I(a) m)]]\|_{L^1} d\t 
\les \langle t\rangle^{-q(d)}(\mD(t) \mX_2(t)+\mD(t)^2). 
\end{aligned}
\end{equation}
For the other nonlinear terms, it can be handled by the same manner as \eqref{est;ref_nonlin1}, \eqref{est;ref_nonlin2}. 
\end{proof}

\begin{proof}[The proof of Proposition \ref{prop;ref_decay2}] 
Here, we only show the estimation of $(m \cdot \N)m$ because the way to the proofs are almost same as the proof of Proposition \ref{prop;decay_low_N}. 
Since $L^1(\R^d) \hr \dB_{2,\infty}^{-d/2}(\R^d)$ and thanks to H\"older's inequality, we obtain that for all $-d/2<s \le 1+d/2$, 
\begin{equation} \label{est;ref12}
\begin{aligned}
\int_0^t \langle t-\t\rangle^{-(\frac{d}{4}+\frac{s}{2})}& \|(m \cdot \N)m_L\|_{\dB_{2,\infty}^{-\frac{d}{2}}} d\t \\
\les& \int_0^t \langle t-\t\rangle^{-(\frac{d}{4}+\frac{s}{2})} \|m\|_{L^2} \|\N m_L\|_{L^2} d\t \\
\les& \mD(t)^2 \int_0^t \langle t-\t\rangle^{-(\frac{d}{4}+\frac{s}{2})} \langle\t\rangle^{-\min(\frac{d}{4},\al)} \langle\t\rangle^{-(\frac{d}{4}+\frac{1}{2})} d\t \\
\les& \langle t\rangle^{-(\frac{d}{4}+\frac{s}{2})} \mD(t)^2. 
\end{aligned}
\end{equation}
When $t \le 2$, it follows from Proposition \ref{prop;unif} and Lemma \ref{lem;cA-P} that 
\begin{equation}
\begin{aligned}
\int_0^t \langle t-\t\rangle^{-(\frac{d}{4}+\frac{s}{2})} \|(m \cdot \N)m_H\|_{\dB_{2,\infty}^{-\frac{d}{2}}} d\t 
\les& \int_0^t \|m\|_{\cB^{-1+\frac{d}{2}}} \|\N m_H\|_{\dB_{2,\infty}^{1-\frac{d}{2}}} d\t \\
\les& \int_0^t \|m\|_{\cB^{-1+\frac{d}{2}}} \|m\|_{\cB^{1+\frac{d}{2}}}^h d\t \\
\les& \mX_2(t)^2. 
\end{aligned}
\end{equation}
When $t>2$, by the similar computation as \eqref{est;ref7}, it plainly holds that 
\begin{equation} \label{est;ref13}
\begin{aligned}
\int_0^t \langle t-\t\rangle^{-(\frac{d}{4}+\frac{s}{2})}& \|(m \cdot \N)m_H\|_{\dB_{2,\infty}^{-\frac{d}{2}}} d\t \\
\les& \langle t\rangle^{-(\frac{d}{4}+\frac{s}{2})} \int_0^1 \left(\frac{\langle t\rangle}{\langle t-\t\rangle}\right)^{\frac{d}{4}+\frac{s}{2}} \|m\|_{\cB^{-1+\frac{d}{2}}} \|\N m_H\|_{\dB_{2,\infty}^{1-\frac{d}{2}}} d\t \\
      &+\int_1^t \langle t-\t\rangle^{-(\frac{d}{4}+\frac{s}{2})} \|m\|_{\cB^{-1+\frac{d}{2}}} \|\N m_H\|_{\dB_{2,\infty}^{1-\frac{d}{2}}} d\t \\
\les& \langle t\rangle^{-(\frac{d}{4}+\frac{s}{2})} \int_0^t \|m\|_{\cB^{-1+\frac{d}{2}}} \|m\|_{\cB^{1+\frac{d}{2}}}^h d\t \\
      &+\mD(t)^2 \int_1^t \langle t-\t\rangle^{-(\frac{d}{4}+\frac{s}{2})} \langle\t\rangle^{-\min(\frac{d}{4},\al)} \langle\t\rangle^{-\al} \left(\frac{\langle\t\rangle}{\t}\right)^\al  d\t \\
\les& \langle t\rangle^{-(\frac{d}{4}+\frac{s}{2})} (\mX_2(t)^2+\mD(t)^2).  
\end{aligned}
\end{equation}
Gathering \eqref{est;ref12}-\eqref{est;ref13}, we get 
\begin{align*}
\int_0^t \langle t-\t\rangle^{-(\frac{d}{4}+\frac{s}{2})}& \|(m \cdot \N)m\|_{\dB_{2,\infty}^{-\frac{d}{2}}} d\t 
\les \langle t\rangle^{-(\frac{d}{4}+\frac{s}{2})} (\mX_2(t)^2+\mD(t)^2). 
\end{align*} 
We would like to conclude the proof of Proposition \ref{prop;ref_decay2}.  
\end{proof}

Applying Proposition \ref{prop;ref_decay1} to \eqref{est;basic_1}, we obtain that 
\begin{align} \label{est;modif_decay1}
 \mD_L^{(2)} \les \|(a_0,m_0)\|_{L^1}+\mX_2(t)^2+\mD(t)^2+\mD(t)^3. 
\end{align} 
Analogously, we obtain by Applying Proposition \ref{prop;ref_decay2} to \eqref{est;basic_2} that  
\begin{align} \label{est;modif_decay2}
 \mD_L^{(1)} \les \|(a_0,m_0)\|_{L^1}+\mX_2(t)^2+\mD(t)^2. 
\end{align} Therefore it is completion to obtain the estimate for $\mD^{(i)}_L(t)$ $(i=1,2)$. 

\subsection{The uniform boundedness for $\mD_H(t)$ and $\mD(t)$}  \label{subsect;bdd2} 
By the same argument as \eqref{decay;high_am}, we are able to obtain that 
\begin{equation}
\begin{aligned}
\|\langle\t\rangle^\al (|\N|a,m)\|_{\wt{L^\infty(I};\cB^{-1+\frac{d}{2}})}^h
&+\|\t^\al (|\N|a,m)\|_{\wt{L^\infty(I};\cB^{-1+\frac{d}{2}})}^h \\
&\les \mX_{2,0}+\|\t^\al N(a,m)\|_{\wt{L^\infty(I};\cB^{-1+\frac{d}{2}})}^h. 
\end{aligned}
\end{equation}
The estimate for the nonlinear term is also same as Proposition \ref{prop;high_N}. 
Therefore, we only give the estimation of $\div ((I(a)-1)m \otimes m)$. 
Thanks to \eqref{est;Banach_ring} and \eqref{est;1/1+r2}, we obtain that  
\begin{align*}
\|\t^\al \div ((I(a)-1)m \otimes m)\|_{\wt{L^\infty_t(}\cB^{-1+\frac{d}{2}})}^h
\les& \|\t^\al (I(a)-1)m \otimes m\|_{\wt{L^\infty_t(}\cB^\frac{d}{2})}^h \\
\les& (1+\mX_2(t)) \times \\ 
      &\left(\|\t^{\frac{d}{2}-\frac{\ve}{2}}m_L\|_{\wt{L^\infty_t(}\cB^\frac{d}{2})}
      \|\t^{\frac{d}{2}-\frac{\ve}{2}}m_L\|_{\wt{L^\infty_t(}\cB^\frac{d}{2})} \right. \\
    &\left. +\|m_L\|_{\wt{L^\infty_t(}\cB^\frac{d}{2})}
      \|\t^\al m_H\|_{\wt{L^\infty_t(}\cB^\frac{d}{2})} \right.  \\
    &\left. +\|\t^\al m_H\otimes m_H\|_{\wt{L^\infty_t(}\cB^\frac{d}{2})}\right) . 
\end{align*}
Applying \eqref{est;KS_type} in Lemma \ref{lem;KS_type} to the last term and using the embedding 
$\cB^{-1+d/2}(\R^d) \hr \fB_{\infty,\infty}^{-1}(\R^d)$ obtained by Lemma \ref{lem;sm}, we see that 
\begin{align*}
\|\t^\al m_H \otimes m_H\|_{\wt{L^\infty_t(}\cB^\frac{d}{2})}
\les \|m_H\|_{\wt{L^\infty_t(}\fB_{\infty,\infty}^{-1})}
      \|\t^\al m_H\|_{\wt{L^\infty_t(}\cfB^{1+\frac{d}{p}})} 
\les \mX_2(t)^2+\mD(t)^2. 
\end{align*}
Combining the above estimates and the similar estimate as \eqref{est;komakai}; 
\begin{align} \label{est;komakai2}
\|\t^{\frac{d}{2}-\frac{\ve}{2}}m_L\|_{\wt{L^\infty_t(}\cB^\frac{d}{2})}
\les \|\langle\t\rangle^{\frac{d}{2}-\frac{\ve}{2}} m_L\|_{L^\infty(I;\cB^{\frac{d}{2}-\ve})} 
\les \mD(t), 
\end{align}
we obtain that 
\begin{align} \label{est;high_ref_N}
\|\t^\al \div ((I(a)+1)m \otimes m)\|_{\wt{L^\infty_t(}\cB^{-1+\frac{d}{2}})}^h 
\les \mX_2(t)^2+\mD(t)^2. 
\end{align}
In a similar way to obtaining \eqref{est;high_ref_N}, one can obtain that 
\begin{align} \label{est;high_ref_N2}
\|\t^\al N(a,m)\|_{\wt{L^\infty_t(}\cB^{-1+\frac{d}{2}})}^h 
\les \mX_2(t)^2+\mD(t)^2 
\end{align}
which yields to the following estimate for the high-frequencies: 
\begin{align} \label{est;modif_decay3}
\mD_H(t) \les \mX_{2,0}+\mX_2(t)^2+\mD(t)^2. 
\end{align}
Therefore, combining \eqref{est;modif_decay1}, \eqref{est;modif_decay2} and \eqref{est;modif_decay3}, we arrive at 
\begin{align} \label{est;modif_decay4}
\mD(t) \les \mX_{2,0}+\|(a_0,m_0)\|_{L^1}+\mX_2(t)^2+\mD(t)^2+\mD(t)^3. 
\end{align}
The above estimate \eqref{est;modif_decay4} together with $\mX_{2,0}+\|(a_0,m_0)\|_{L^1} \ll 1$ leads to the our assertion 
$\mD(t) \les 1$. From this uniform estimate, we immediately obtain $\|a(t)\|_{L^\infty}=O(t^{-q(d)})$ as $t \to \infty$. 

\subsection{The time-decay of the differences $m(t)-e^{t \mu \Delta} \mathcal{P}_\s m_0$} 
In order to complete the proof of Theorem \ref{thm;asympt}, 
we need to show the estimate for the difference between the momentum and the solution to the Stokes equation 
$m(t)-e^{t \mu\Delta} \mathcal{P}_\s m_0$ to use $\mD(t) \les 1$ as obtained in \S \ref{subsect;bdd2}. 
In what follows, we set $m_c(t):=m(t)-e^{t \mu \Delta} \mathcal{P}_\s m_0$ for simplicity. 
Firstly, let us consider the estimate for the $m_c$ in the low-frequency region. 
By $\eqref{eqn;aLmL}_2$, the low-frequency part of $m$ can be written down the followings: 
\begin{align*}
m_{c,L}(t)=G_L^{2,1}(t,\cdot)*a_0+\wt{G_L^{2,2}}(t,\cdot)*m_0+\int_0^t G_L^{2,2}(t-\t,\cdot)*N(a,m)\,d\t, 
\end{align*}
where $m_{c,L}=m_{c,L}(t)$ and $\wt{G^{2,2}_L}(t,x)$ are given by 
$$
m_{c,L}=\F^{-1}[\chi_L \wh{m_c}], \quad 
\wt{G^{2,2}_L}(t,x)=\F^{-1}\left[\chi_L(\xi) \frac{\lam_+(\xi)e^{\lam_+(\xi)t}-\lam_-(\xi)e^{\lam_-(\xi)t}}{\lam_+(\xi)-\lam_-(\xi)} \cdot \frac{\xi {}^t\xi}{|\xi|^2}\right]. 
$$
By using Lemma \ref{lem;DW} and noting that 
$$
G^{2,1}_{L}(t,\cdot)=-\N(\gm^2+\kappa \Delta) K_{L,\psi_1}(t,\cdot), \quad \wt{G_L^{2,2}}(t,\cdot)=\pt_t K_{L,\psi_2}(t,\cdot) 
$$
with $\psi_1 \equiv 1$ and $\psi_2=\frac{\xi {}^t\xi}{|\xi|^2}$, 
we immediately see that 
\begin{align} \label{est;fund_2}
\|G_L^{2,1}(t,\cdot)\|_{L^\infty}+\|\wt{G^{2,2}_L}(t,\cdot)\|_{L^\infty} \les \langle t\rangle^{-\frac{(3d-1)}{4}}. 
\end{align}
In a similar way to getting \eqref{est;basic_1}, we obtain by the Hausdorff-Young inequality, \eqref{est;fund_2} 
and the boundedness of the Helmholtz projection $\mathcal{P}_\s$ on $\dot{B}_{2,1}^s$ that 
\begin{equation} \label{est;fund_mc}
\begin{aligned}
\|m_{c,L}(t)\|_{L^\infty}
\les& \|G_L^{2,1}(t,\cdot)\|_{L^\infty}\|a_0\|_{L^1}+\|\wt{G^{2,2}_L}(t,\cdot)\|_{L^\infty}\|m_0\|_{L^1} \\
     &+\int_0^t \|\wt{G^{2,2}_L}(t-\t,\cdot)\|_{L^\infty} \|\wt{N_L}(a,m)\|_{L^1}\,d\t \\
     &+\int_0^t \|e^{(t-\t)\mu \Delta}\mathcal{P}_\s N_L(a,m)\|_{L^\infty}d\t \\
\les& \langle t\rangle^{-\frac{(3d-1)}{4}}\|(a_0,m_0)\|_{L^1}+\int_0^t \langle t-\t\rangle^{-\frac{(3d-1)}{4}}  \|\wt{N_L}(a,m)\|_{L^1}\,d\t \\
     &+\int_0^t \|e^{(t-\t)\mu \Delta} N_L(a,m)\|_{\cB^\frac{d}{2}}d\t. 
\end{aligned}
\end{equation}
As for the last term, it follows from the arguments similar to \eqref{est;G*U_0} and \eqref{est;tl_I2} that 
\begin{equation} \label{est;fund_mc2}
\begin{aligned}
\int_0^t \|e^{(t-\t)\mu \Delta} N_L(a,m)\|_{\cB^\frac{d}{2}}d\t 
\les& \int_0^{t/2} \langle t-\t\rangle^{-(\frac{d}{2}+\frac{1}{2})} \||\N|^{-1} N_L(a,m)\|_{\dB_{2,\infty}^{-\frac{d}{2}}} d\t \\
      &+\int_{t/2}^t \|N_L(a,m)\|_{\cB^\frac{d}{2}} d\t. 
\end{aligned}
\end{equation}

\begin{prop} \label{prop;ref_decay3} Let $d \ge 2$.  
There exists some constant $C>0$ such that for all $t>2$, 
\begin{align*}
 \int_0^t \langle t-\t\rangle^{-(\frac{d}{2}+\frac{1}{2})} \||\N|^{-1} N_L(a,m)\|_{\dB_{2,\infty}^{-\frac{d}{2}}} d\t 
\le Ct^{-(\frac{d}{2}+\frac{1}{2})} \delta(t), 
\end{align*}
where we recall that $\delta(t)$ is given by $\log t$ if $d=2$ and $1$ if $d \ge 3$. 
\end{prop}

\begin{prop} \label{prop;ref_decay4} Let $d \ge 2$.  
There exists some constant $C>0$ such that for all $t>2$, 
\begin{align*}
\int_{t/2}^t \|N_L(a,m)\|_{\cB^\frac{d}{2}} d\t
\le Ct^{-(\frac{d}{2}+\frac{1}{2})}. 
\end{align*}
\end{prop}

\begin{proof}[The proof of Proposition \ref{prop;ref_decay3}]
Here, we only show the estimation of $\div((I(a)-1)m \otimes m)$ 
because the way to the proofs are almost same as the proof of Proposition \ref{prop;0_t/2}.  
Thanks to \eqref{est;decay}, 
\eqref{est;1/1+r}, \eqref{est;p-3} and \eqref{est;p-4}, we see that 
\begin{equation} \label{0_t/2;convect_a}
\begin{aligned}
\int_0^{t/2}
       \langle t-\t\rangle^{-(\frac{d}{2}+\frac{1}{2})}
       & \||\N|^{-1}\div((I(a)-1)m \otimes m)\|_{\dB_{2,\infty}^{-\frac{d}{2}}}
      d\t \\
\les& \int_0^{t/2}
       \langle t-\t\rangle^{-(\frac{d}{2}+\frac{1}{2})}
       \Big(1+\|a\|_{\cB^\frac{d}{2}}\Big) \|m \otimes m\|_{\dB_{2,\infty}^{-\frac{d}{2}}}
      d\t \\
\les& (1+\mX_2(t)) \int_0^{t/2}
       \langle t-\t\rangle^{-(\frac{d}{2}+\frac{1}{2})}
       \|m\|_{\cB^{-1+\frac{d}{2}}} \|m\|_{\dB_{2,\infty}^{1-\frac{d}{2}}}
      d\t \\
\les& t^{-(\frac{d}{2}+\frac{1}{2})} (1+\mX_2(t)) \mD(t)^2 
      \int_0^{t/2} 
      \langle\t\rangle^{-\frac{d}{2}}
      d\t
 \les t^{-(\frac{d}{2}+\frac{1}{2})} \delta(t). 
\end{aligned}
\end{equation}
Analogously, we are able to obtain the desired estimate. 
\end{proof}

\begin{proof}[The proof of Proposition \ref{prop;ref_decay4}] 
Here, we only show the estimation of $\div((I(a)-1)m \otimes m)$ 
because the way to the proofs are almost same as the proof of Proposition \ref{prop;t/2_t}. 
Notice that 
\begin{align*}
\div ((I(a)-1)m \otimes m)
= (\N I(a) \cdot m)m+(I(a)-1)\,((m \cdot \N)m+m \,\div m), 
\end{align*}
it follows from Lemma \ref{lem;A-P} with \eqref{est;Banach_ring}, \eqref{est;Na} 
and \eqref{est;asert_decay} that 
\begin{align*}
\int_{t/2}^t \|(\N I(a) \cdot m)m\|_{\cB^\frac{d}{2}}d\t
\les& \int_{t/2}^t
      \|\N a\|_{\cB^\frac{d}{2}}
      \|m\|_{\cB^\frac{d}{2}}\|m\|_{\cB^\frac{d}{2}} d\t \\
\les& \mD(t)^3
      \int_{t/2}^t
       \t^{-(\frac{d}{2}+\frac{1}{2})} \t^{-d}
       d\t \\
\les& t^{-(\frac{d}{2}+\frac{1}{2})}t^{-d} \frac{t}{2}
\les t^{-(\frac{d}{2p'}+\frac{s+1}{2})}, 
\end{align*}
where we used $d \ge 2$ in the last inequality. Similarly, 
by using \eqref{est;Banach_ring} and \eqref{est;1/1+r}, 
\begin{align*}
\int_{t/2}^t \|(I(a)-1)(m\cdot\N)m\|_{\cB^\frac{d}{2}}d\t
\les& \int_{t/2}^t
      \Big(1+\|a\|_{\cB^\frac{d}{2}}\Big)
      \|m\|_{\cB^\frac{d}{2}} \|\N m\|_{\cB^\frac{d}{2}} d\t \\
\les& (1+\mX_2(t)) \mD(t)^2
      \int_{t/2}^t
        \t^{-(d+\frac{1}{2})} 
       d\t 
\les t^{-(\frac{d}{2}+\frac{1}{2})}
\end{align*}
and 
\begin{align*}
\int_{t/2}^t \|(I(a)-1)\,m \,\div m\|_{\cfB^s}d\t
\les t^{-(\frac{d}{2}+\frac{1}{2})}. 
\end{align*}
Combining the above estimations, we obtain 
\begin{align} \label{t/2_t;convect}
\int_{t/2}^t \|\div ((I(a)-1)m \otimes m)\|_{\cfB^s}d\t
\les t^{-(\frac{d}{2}+\frac{1}{2})}. 
\end{align}
Analogously, we are able to obtain the desired estimate. 
\end{proof}

Applying Propositions \ref{prop;ref_decay1}, \ref{prop;ref_decay3} and \ref{prop;ref_decay4} to \eqref{est;fund_mc} with \eqref{est;fund_mc2},  we get 
\begin{align} \label{est;low_mc} 
\|m_{c,L}(t)\|_{L^\infty} \les t^{-q(d)}. 
\end{align}

Finally, let us consider the estimate for the high-frequency part of $m_c$. 
Notice that $\pt_t \mathcal{P}_\s e^{t \mu\Del} m_0-\mu \Del \mathcal{P}_\s e^{t \mu\Del} m_0=0$ and $\div \mathcal{P}_\s e^{t \mu\Del} m_0=0$, 
\eqref{eqn;mNSK2} can be transformed into the following system:
\begin{equation} \label{eqn;m_c}
\left\{
\begin{aligned}
&\pt_t a+\div m_{c}=0, \\
&\pt_t m_c-\mu \Del m_c+(\lam+\mu)\N \div m_c+\gm^2\N a - \kappa \N \Delta a=N(a,m), \\
&(a,m_c)|_{t=0}=(a_0,\mathcal{P}_\s^{\perp} m_0), 
\end{aligned}
\right. 
\end{equation}
where $\mathcal{P}_\s^{\perp}:=\text{Id}-\mathcal{P}_\s$. 
By the same argument as \eqref{decay;high_am}, we are able to obtain that 
\begin{equation} \label{est;high_mc}
\begin{aligned}
\|\langle\t\rangle^\al (|\N|a,m_c)\|^h_{\wt{L_t^\infty}(\cB^{-1+\frac{d}{2}})}
&+\|\t^\al (|\N|a,m_c)\|^h_{\wt{L_t^\infty}(\cB^{1+\frac{d}{2}})} \\
&\les \|(|\N|a_0,\mathcal{P}_\s^\perp m_0)\|_{\cB^{-1+\frac{d}{2}}}^h+\|\t^\al N(a,m)\|_{\wt{L^\infty_t}(\cB^{-1+\frac{d}{2}})}^h. 
\end{aligned}
\end{equation}
By \eqref{est;high_ref_N2} and $\mX_2(t) \les 1$, $\mD(t) \les 1$, we obtain from \eqref{est;high_mc} that 
\begin{align} \label{est;high_mc2}
\|\langle\t\rangle^\al (|\N|a,m_c)\|^h_{\wt{L_t^\infty}(\cB^{-1+\frac{d}{2}})}+\|\t^\al (|\N|a,m_c)\|^h_{\wt{L_t^\infty}(\cB^{1+\frac{d}{2}})}
\les 1. 
\end{align}
Therefore, by \eqref{est;low_mc} and \eqref{est;high_mc}, we obtain that for all $t>2$, 
\begin{equation} \label{est;mc}
\begin{aligned}
\|m_c(t)\|_{L^\infty} \les& \|m_{c,L}(t)\|_{L^\infty}+\|m_{c,H}(t)\|_{L^\infty} \\
                               \les& \|m_{c,L}(t)\|_{L^\infty}+\|m_c(t)\|_{\cB^{1+\frac{d}{2}}}^h 
                               \les t^{-q(d)}+t^{-\al} \les t^{-q(d)}. 
\end{aligned}
\end{equation}
Therefore, the conclusion of \S \ref{subsect;bdd2} and \eqref{est;mc} give us the assertion of Theorem \ref{thm;asympt}. 
\vskip2mm
\noindent
{\bf Acknowledgments}. 
The first author is supported by Grant-in-Aid for Scientific Research (C) JP22K03374. 
The second author is supported by JSPS Grant-in-Aid for Early-Career Scientists JP22K13936. 


\end{document}